\pdfoutput=1
\documentclass{amsart}
\usepackage{amscd}
\usepackage{amssymb}
\usepackage{graphicx}

\makeatletter

\@addtoreset{equation}{section}
\makeatother

\theoremstyle{plain}
\newtheorem{theorem}{Theorem}
\newtheorem{proposition}{Proposition}
\newtheorem{lemma}{Lemma}

\newtheorem{corollary}{Corollary}

\theoremstyle{definition}
\newtheorem{definition}{Definition}
\newtheorem{example}{Example}
\newtheorem{notation}{Notation}

\theoremstyle{remark}

\newcommand{\four}{
\,\begin{picture}(8,8)
\put(5,3){\circle{10}}
\qbezier(0,2)(5,2)(10,2)
\qbezier(1,0)(5,0)(9,0)
\qbezier(0,3.5)(5,3.5)(10,3.5)
\qbezier(0.5,5)(5,5)(9.5,5)
\end{picture}~
}
\newcommand{\triFour}{
\,\begin{picture}(8,8)
\put(5,3){\circle{10}}
\qbezier(4,8)(2,6)(0,4)
\qbezier(4,-2)(2,0)(0,2)
\qbezier(6,8)(8,6)(10,4)
\qbezier(6,-2)(8,0)(10,2)
\end{picture}~
}
\newcommand{\trisox}{
\,\begin{picture}(8,8)
\put(5,3){\circle{10}}
\qbezier(0.5,2)(5,3)(10,4)
\qbezier(0.5,4)(5,3)(9.5,2)
\qbezier(1.5,6)(5,6)(8.5,6)
\qbezier(1.5,0)(5,0)(8.5,0)
\end{picture}~
}
\newcommand{\chordtwx}{\,
\begin{picture}(8,8)
\put(5,3){\circle{10}}
\qbezier(9,0)(9,3.5)(9,6)
\qbezier(1,0)(4.5,3.5)(8,7)
\qbezier(1.5,6.6)(5,3.3)(8.5,0)
\qbezier(2,-1)(5,-1)(8,-1)
\end{picture}~
}

\newcommand{\triii}{
\,\begin{picture}(8,8)
\put(5,3){\circle{10}}
\qbezier(1,0)(1,3)(1,6)
\qbezier(2,-1)(6,1.5)(10,4)
\qbezier(4,-2)(7,0)(10,2)
\qbezier(7,-1.5)(7,3)(7,7.5)
\end{picture}~
}

\newcommand{\trii}{
\,\begin{picture}(8,8)
\put(5,3){\circle{10}}
\qbezier(2,-1)(2,3)(2,7)
\qbezier(4,8)(7,6)(10,4)
\qbezier(4,-2)(7,0)(10,2)
\qbezier(7,-1)(7,3)(7,7)
\end{picture}~
}
\newcommand{\chordtwoc}{
\,\begin{picture}(8,8)
\put(5,3){\circle{10}}
\qbezier(1,0)(5,1)(10,2)
\qbezier(0.5,2)(5,1)(9,0)
\qbezier(0.5,5)(4,6)(8,7)
\qbezier(2,7)(5,6)(9.5,5)
\end{picture}~
}
\newcommand{\triso}{
\,\begin{picture}(8,8)
\put(5,3){\circle{10}}
\qbezier(1,0)(5,1)(10,2)
\qbezier(0.5,2)(5,1)(9,0)
\qbezier(0.5,1)(5,1)(9.5,1)
\qbezier(0.5,5)(5,5)(9.5,5)
\end{picture}~
}
\newcommand{\chordtha}{
\,\begin{picture}(8,8)
\put(5,3){\circle{10}}
\qbezier(1,0)(5,1)(10,2)
\qbezier(0.5,2)(5,1)(9,0)
\qbezier(0,3.5)(5,3.5)(10,3.5)
\qbezier(0.5,5)(5,5)(9.5,5)
\end{picture}~
}
\newcommand{\chordthab}{
\,\begin{picture}(8,8)
\put(5,3){\circle{10}}
\qbezier(4,-2)(2,0)(0,2)
\qbezier(10,2)(8,0)(6,-2)
\qbezier(0,3.5)(5,3.5)(10,3.5)
\qbezier(0.5,5)(5,5)(9.5,5)
\end{picture}~
}

\newcommand{\chordthb}{
\,\begin{picture}(8,8)
\put(5,3){\circle{10}}
\qbezier(1,0)(5,1)(10,2)
\qbezier(0.5,2)(5,1)(9,0)
\qbezier(0,4)(2,6)(4,7.5)
\qbezier(10,4)(7.5,6)(5,8)
\end{picture}~
}

\newcommand{\ma}{\,
\begin{picture}(8,8)
\put(5,3){\circle{10}}
\qbezier(6,-2)(6,3)(6,8)
\qbezier(3,7)(3,3)(3,-2)
\qbezier(1,0)(4.5,3.5)(8,7)
\qbezier(1.5,6.6)(5,3.3)(8.5,0)
\end{picture}~
}

\newcommand{\sh}{
\,\begin{picture}(8,8)
\put(5,3){\circle{10}}
\qbezier(8,-1)(8,3)(8,7)
\qbezier(1.5,6.5)(1.5,3)(1.5,-0.5)
\qbezier(0,3.5)(5,3.5)(10,3.5)
\qbezier(0,1)(5,1)(9,1)
\end{picture}~
}
\newcommand{\hthc}{
\,\begin{picture}(8,8)
\put(5,3){\circle{10}}
\qbezier(8,-1)(8,3)(8,7)
\qbezier(5,-2)(5,3)(5,8)
\qbezier(1.5,6.5)(1.5,3)(1.5,-0.5)
\qbezier(0,3.5)(5,3.5)(10,3.5)
\end{picture}~~~
}
\newcommand{\crh}{\,
\begin{picture}(8,8)
\put(5,3){\circle{10}}
\qbezier(1,0)(4.5,3.5)(8,7)
\qbezier(1.5,6.6)(5,3.3)(8.5,0)
\qbezier(4,8)(7,6)(10,4)
\qbezier(4,-2)(7,0)(10,2)
\end{picture}~
}

\newcommand{\htri}{\,
\begin{picture}(8,8)
\qbezier(5,-2)(5,3)(5,8)
\put(5,3){\circle{10}}
\qbezier(1,0)(4.5,3.5)(8,7)
\qbezier(1.5,6.6)(5,3.3)(8.5,0)
\qbezier(2,-1)(5,-1)(8,-1)
\end{picture}~
}
\newcommand{\qua}{
\,\begin{picture}(8,8)
\put(5,3){\circle{10}}
\qbezier(1,0)(4.5,3.5)(8,7)
\qbezier(1.5,6.6)(5,3.3)(8.5,0)
\qbezier(4.5,-2)(4.5,3)(4.5,8)
\qbezier(0,3.5)(5,3.5)(10,3.5)
\end{picture}~
}

\newcommand{\Fsum}{\sum_{n_{b-1} + 1 \le i \le n_{d}} \alpha_i x_i}
\newcommand{\aFsum}{\sum_{n_{b-1} +1 \le i \le n_{d}} \alpha_i  \tilde{x}_i  }
\newcommand{\Nsum}{ \sum_{n_{b-1} + 1 \le i \le n_d}
}
\newcommand{\Njsum}{ \sum_{n_{b-1} + 1 \le j \le n_d}
}

\newcommand{\at}{\alpha_i \tilde{x}_i}

\newcommand{\sFsum}{\sum_{n_{b-1} + 1 \le i \le n_{d}} \alpha_i}

\newcommand{\aFsumIrr}{\sum_{i \in \irri} \alpha_i \tilde{x}_i}

\newcommand{\FsumIrr}{\sum_{i \in \irri} \alpha_i x_i}

\newcommand{\FsumConn}{\sum_{i \in \conni} \alpha_i x_i}
\newcommand{\aFsumConn}{\sum_{i \in \conni} \alpha_i \tilde{x}_i 
}

\newcommand{\ineq}{n_{b-1} + 1 \le i \le n_{d}}

\newcommand{\cyc}{\operatorname{cyc}}
\newcommand{\rev}{\operatorname{rev}}
\newcommand{\Rep}{R_{\epsilon_1 \epsilon_2 \epsilon_3 \epsilon_4 \epsilon_5}}

\newcommand{\x}{\,
\begin{picture}(8,8)
\put(5,3){\circle{10}}
\qbezier(1,0)(4.5,3.5)(8,7)
\qbezier(1.5,6.6)(5,3.3)(8.5,0)
\end{picture}~
}

\newcommand{\tr}{
\,\begin{picture}(8,8)
\put(5,3){\circle{10}}
\qbezier(1,0)(4.5,3.5)(8,7)
\qbezier(1.5,6.6)(5,3.3)(8.5,0)
\qbezier(4.5,-2)(4.5,3)(4.5,8)
\end{picture}~
}

\newcommand{\tri}{
\,\begin{picture}(8,8)
\put(5,3){\circle{10}}
\qbezier(2,-1)(2,3)(2,7)
\qbezier(4,8)(7,6)(10,4)
\qbezier(4,-2)(7,0)(10,2)
\end{picture}~
}

\newcommand{\thc}{
\,\begin{picture}(8,8)
\put(5,3){\circle{10}}
\qbezier(8,-1)(8,3)(8,7)
\qbezier(5,-2)(5,3)(5,8)
\qbezier(1.5,6.5)(1.5,3)(1.5,-0.5)
\end{picture}~~~
}

\newcommand{\h}{
\,\begin{picture}(8,8)
\put(5,3){\circle{10}}
\qbezier(8,-1)(8,3)(8,7)
\qbezier(1.5,6.5)(1.5,3)(1.5,-0.5)
\qbezier(0,3.5)(5,3.5)(10,3.5)
\end{picture}~
}

\newcommand{\chordtwo}{\,
\begin{picture}(8,8)
\put(5,3){\circle{10}}
\qbezier(1,0)(4.5,3.5)(8,7)
\qbezier(1.5,6.6)(5,3.3)(8.5,0)
\qbezier(2,-1)(5,-1)(8,-1)
\end{picture}~
}

\newcommand{\chordth}{
\,\begin{picture}(8,8)
\put(5,3){\circle{10}}
\qbezier(8,-1)(8,3)(8,7)
\qbezier(1.5,6.5)(1.5,3)(1.5,-0.5)
\end{picture}~
}

\newcommand{\conn}{\operatorname{Conn}}
\newcommand{\irr}{\operatorname{Irr}}
\newcommand{\irri}{{I}^{(\operatorname{Irr})}_{b, d}}
\newcommand{\irritwo}{{I}^{(\operatorname{Irr})}_{2, 3}}

\newcommand{\conni}{I^{(\operatorname{Conn})}_{b, d}}

\newcommand{\sub}{\operatorname{Sub}}
\newcommand{\ii}{{\rm{I}}}
\newcommand{\sss}{{\rm{SI\!I\!I}}}
\newcommand{\www}{{\rm{WI\!I\!I}}}
\newcommand{\s}{{\rm{SI\!I}}}
\newcommand{\w}{{\rm{WI\!I}}}

\begin{document}
\title[Space of chord diagrams on spherical curves]{Space of chord diagrams on spherical curves}
\author{Noboru Ito}
\address{Graduate School of Mathematical Sciences, The University of Tokyo, 3-8-1, Komaba, Meguro-ku, Tokyo 153-8914, Japan}
\email{noboru@ms.u-tokyo.ac.jp}
\maketitle

\begin{abstract}
In this paper, we give a definition of $\mathbb{Z}$-valued functions from the ambient isotopy classes of spherical/plane curves derived from chord diagrams, denoted by $\sum_i \alpha_i x_i$.  Then, we introduce certain elements of the free $\mathbb{Z}$-module generated by the chord diagrams with at most $l$ chords, called relators of Type~(\ii) ((\s), (\w), (\sss), or (\www), resp.), and introduce another function $\sum_i \alpha_i \tilde{x}_i$ derived from $\sum_i \alpha_i x_i$.  The main result (Theorem~\ref{gg_thm1}) shows that if $\sum_i \alpha_i \tilde{x}_i$ vanishes for the relators of Type~(\ii) ((\s), (\w), (\sss), or (\www), resp.), then $\sum_i \alpha_i x_i$ is invariant under the Reidemeister move of type RI (strong~RI\!I, weak~RI\!I, strong~RI\!I\!I, or weak~RI\!I\!I, resp.) that is defined in \cite{IT_weak13}.
\end{abstract}


\section{Introduction}\label{intro}
A spherical curve (plane curve, resp.) is the image of a  generic immersion of a circle into a $2$-sphere (plane, resp.). In this paper, we study certain equivalence classes of spherical/plane curves.   
Any two spherical/plane curves can be transformed into each other by a finite sequence of Reidemeister moves, each of which is either one of types RI, RI\!I, or RI\!I\!I that is a replacement of a part of the curve as Figure~\ref{reidemei}.  
\begin{figure}[h!]
\includegraphics[width=11cm]{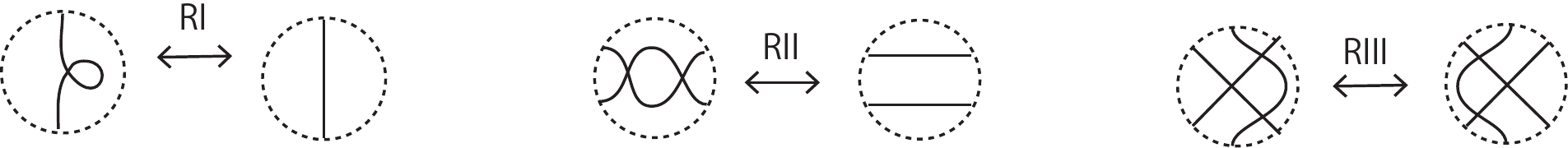}
\caption{Reidemeister moves: types RI, RI\!I, and RI\!I\!I.}\label{reidemei}
\end{figure}
These moves are obtained from Reidemeister moves of types $\Omega_1$, $\Omega_2$, and $\Omega_3$ on knot diagrams by ignoring over/under information.  It is well-known that two knot diagrams represent the same knot if and only if one knot diagram is transformed into the other knot diagram by a finite sequence of Reidemeister moves of types $\Omega_1$, $\Omega_2$, and $\Omega_3$.    

In 1990, Vassiliev \cite{Va} defined an infinite sequence of knot invariants that are at least as powerful as all of the quantum group invariants of knots. 
After that some researchers gave alternative  definitions of Vassiliev invariants (see \cite{BT}).  
In 1994, Polyak and Viro \cite{PV} introduced a formulation of Vassiliev invariants using chord diagrams, and presented explicit formulas for some Vassiliev invariants.  Goussarov \cite{G} proved that the formulation of Polyak-Viro provides all Vassiliev invariants of knots.  In 2001, \"{O}stlund further developed the theory of Polyak and Viro by using singularity theory of plane curves,  and obtained the following result \cite[Theorem~6]{ostlund}: if $v$ is a function on the set of knot diagrams that is invariant under plane isotopy, $\Omega_1$, and $\Omega_3$ and is so-called Vassiliev-type, then $v$ is invariant under $\Omega_2$.  
Motivated by this result, \"{O}stlund \cite{ostlund} posed the following question: can every plane curve be transformed into the simple closed curve by a finite sequence of Reidemeister moves of types RI and RI\!I\!I?

In 2008, Hagge and Yazinski \cite{HY} showed that the answer to this question is negative, i.e., they showed that the equivalence classes under RI and RI\!I\!I are nontrivial.  
In fact, any finite sequence of Reidemeister moves that transforms the plane curve depicted in Figure~\ref{hy} into the simple closed curve contains a Reidemeister move of type RI\!I.  
\begin{figure}[h!]
\includegraphics[width=2cm]{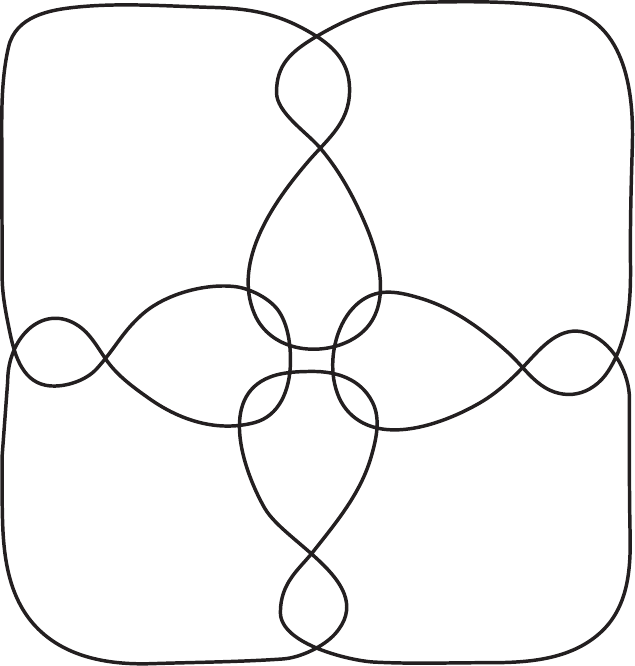}
\caption{Example by Hagge and Yazinski \cite{HY}.}\label{hy}
\end{figure}
However, to the best of the author's knowledge, 
the equivalence class containing the simple closed curve under RI and RI\!I\!I has not been  determined yet.  
In general, none of the equivalence classes under RI and RI\!I\!I have been determined yet.    
   
Viro \cite{Viro1996} suggested the idea that type RI\!I\!I is decomposed into the following two types:   
suppose that a plane curve $P_1$ is transformed into $P_2$ by a single RI\!I\!I.   In this paper, we apply the terminology to spherical curves.  
Note that in RI\!I\!I in Figure~\ref{reidemei}, a triangle is observed in each of the disks.  
We say that spherical/plane curves $P_1$ and $P_2$ are \emph{related by a strong~RI\!I\!I} if the orientations on the edges of the triangle induced by an orientation of the spherical/plane curve are coherent.  If $P_1$ and $P_2$ are not related by a strong~RI\!I\!I, then we say that $P_1$ and $P_2$ are \emph{related by a weak~RI\!I\!I}.  
Analogously, for type RI\!I, we define strong~RI\!I and weak RI\!I as follows.  Suppose that $P_1$ is transformed into $P_2$ by a single RI\!I.  We say that $P_1$ and $P_2$ are \emph{related by a strong~RI\!I} if the orientations on the edges of the triangle induced by an orientation of the spherical/plane curve are coherent.  If $P_1$ and $P_2$ are not related by a strong~RI\!I, then we say that $P_1$ and $P_2$ are \emph{related by a weak~RI\!I}.  
See Figure~\ref{fineReide}.  
\begin{figure}[h!]
\includegraphics[width=10cm]{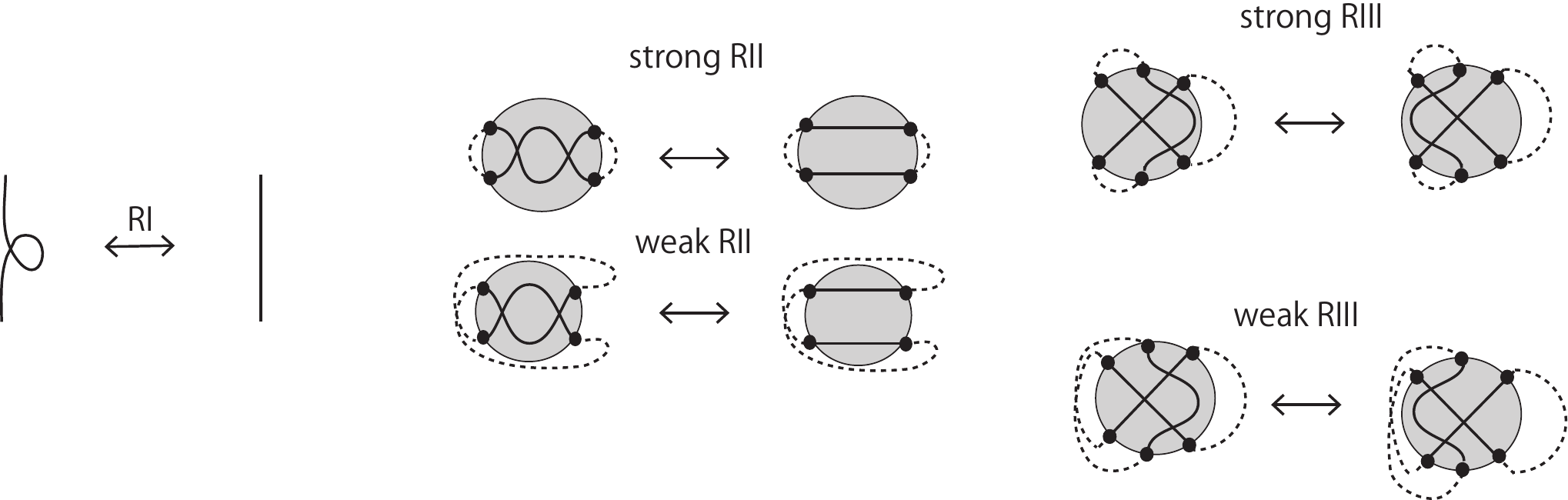}
\caption{RI, strong~RI\!I, weak~RI\!I, strong~RI\!I\!I, and weak~RI\!I\!I.}\label{fineReide}
\end{figure}

From the viewpoint of Viro,  the following problems arise from \"{O}stlund's question above: which spherical curve is transformed into the simple closed curve by a finite sequence of Reidemeister moves of types RI and  strong~RI\!I\!I (types RI and weak~RI\!I\!I, resp.)?      
First, in 2013, a necessary and sufficient condition that a spherical curve and the simple closed curve are related by a finite sequence of Reidemeister moves of types RI and weak~RI\!I\!I was given \cite{IT_weak13}.  
Second, in 2015, a necessary and sufficient condition that a spherical curve and the simple closed curve are related by a finite sequence of Reidemeister moves of types RI and strong~RI\!I\!I was given \cite{ITT}.  
After that, in 2015, different necessary and sufficient conditions using invariants were given \cite{IT3}.  
More concretely, for a chord diagram corresponding to a spherical curve, by counting the number of sub-chord diagrams of certain types, integer-valued functions $\lambda, H,$ and $\widetilde{X}$ were defined, and the following was shown: a spherical curve $P$ and the simple closed curve are related by a finite sequence of Reidemeister moves of types RI and  strong~RI\!I\!I (types RI and weak~RI\!I\!I, resp.) if and only if $H(P)=0$ and $\lambda(P)=0$ ($\widetilde{X}(P)$ $=$ $0$, resp.) \cite{IT3}.     

In this paper, we further develop the idea of counting sub-chord diagrams, which leads to functions of all equivalence classes of spherical/plane curves generated by all possible combinations of Reidemeister moves of types RI, strong~RI\!I, weak~RI\!I, strong~RI\!I\!I, and  weak~RI\!I\!I.  
In \cite{PV}, for a chord diagram $A$, and a chord diagram $G$ obtained from a knot diagram, Polyak and Viro define an integer $\langle A, G \rangle$.  
They further consider a linear combination of $\langle A, \cdot \rangle$, that is denoted by $\langle \sum_i r_i A_i, \cdot \rangle$, which is a function on knot diagrams, and show that this type of the function induces a non-trivial Vassiliev knot invariant.    
We will mimic their idea to define a function, denoted by $\sum \alpha_i x_i$, where $x_i$ denotes a chord diagram.  
Roughly speaking, $x_i (P)$ is obtained from $\langle A, G \rangle$ by ignoring ``signs and arrows'', and this means that we ignore the over/under information of crossings and orientations of knot diagrams.   
In Section~\ref{s_def}, we introduce certain elements of the free $\mathbb{Z}$-module $\mathbb{Z}[G_{\le l}]$ generated by the chord diagrams with at most $l$ chords, called relators of Type~(\ii) ((\s), (\w), (\sss), or (\www), resp.), and introduce another function $\sum_i \alpha_i \tilde{x}_i$ derived from $\sum_i \alpha_i x_i$.  
The main result (Theorem~\ref{gg_thm1}) shows that if $\sum_i \alpha_i \tilde{x}_i$ vanishes for the relators of Type~(\ii) ((\s), (\w), (\sss), or (\www), resp.), then $\sum_i \alpha_i x_i$ is invariant under RI (strong~RI\!I, weak~RI\!I, strong~RI\!I\!I, or weak~RI\!I\!I, resp.).  
In Section~\ref{sec_appl}, we show that Theorem~\ref{gg_thm1} produces a new function that is invariant under RI and strong~RI\!I\!I and is able to distinguish some pairs which the function $\lambda$ of \cite{IT3} cannot distinguish.

\section{Preliminaries and main results}\label{s_def}
\subsection{Definitions and notations}\label{section_def1}  
\begin{definition}[Gauss word]\label{d1}
A {\it{word}} $w$ of length $n$ is a map from $\hat{n}$ $=$ $\{1, 2, 3, \dots, n\}$ to $\mathbb{N}$.  This word is represented by $w(1)w(2)w(3) \cdots w(n)$.  
For a word $w : \hat{n}$ $\to$ $\mathbb{N}$, we call each element of $w(\hat{n})$ a {\it{letter}}.  
A word $u$ of length $q$ is a \emph{sub-word} of $w$ if there is an integer $p$ ($q \le p \le n$) such that $u(j)$ $=$ $w(n-p+j)$ ($1 \le j \le q$).    
A {\it{Gauss word}} of length $2n$ is a word $w$ of length $2n$ where each letter in $w(\hat{2n})$ appears exactly twice in $w(1)w(2)w(3) \cdots w(2n)$.    
Let $\cyc$ and $\rev$ be maps $\hat{2n} \to \hat{2n}$ such that $\cyc(p) \equiv p+1$ (mod $2n$) and $\rev(p) \equiv -p+1$ (mod $2n$).   
Two Gauss words, $v$ and $w$, of length $2n$ are {\it{isomorphic}} if there exists a bijection $f : v(\hat{2n})$ $\to$ $w(\hat{2n})$ satisfying the following: there exists $t \in \mathbb{Z}$ such that $w \circ (\cyc)^t \circ (\rev)^{\epsilon} = f \circ v$ ($\epsilon = 0$ or $1$).  The isomorphisms give an equivalence relation on the Gauss words.  
For a Gauss word $v$ of length $2n$, $[v]$ denotes the equivalence class containing $v$.   A Gauss word $v'$ is a {\emph{sub-Gauss word}} of the Gauss word $v$ if $v'$ is obtained from $v$ by ignoring some letters of $v$.  Then $\sub(v)$ denotes the set of sub-Gauss words of $v$.    
\end{definition}
\begin{definition}[chord diagram]\label{dfn2} 
Each configuration of $n$ pair(s) of points on a circle up to ambient isotopy and reflection of the circle is called a {\it{chord diagram}}.  The integer $n$ is called the {\it{length}} of the chord diagram.  Traditionally, two points of each pair are connected by a straight arc, called a {\it{chord}}.  
\end{definition}

We note that the equivalence classes of Gauss words of length $2n$ have one to one correspondence with the chord diagrams, each of which has $n$ chords as in Figure~\ref{four}.     
\begin{figure}[htbp]
\begin{center}
\includegraphics[width=10cm]{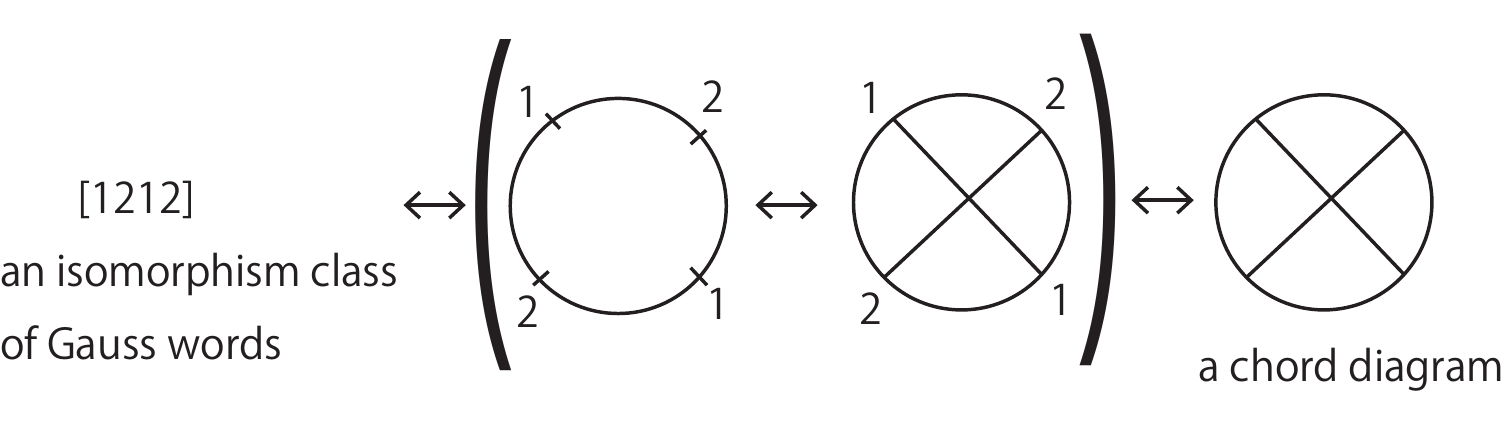}
\caption{Four expressions.}
\label{four}
\end{center}
\end{figure}
In the rest of this paper, we identify these four expressions in Figure~\ref{four}, and freely use either one of them depending on the situation.  
\begin{notation}[$G_{\le d}$, $n_d$, $G_{b, d}$]\label{not2}
Let $G_{< \infty}$ be the set of  chord diagrams, that is, the set of isomorphism classes of Gauss words.  It is clear that $G_{< \infty}$ consists of countably many elements.  Hence, there exists a bijection between $G_{< \infty}$ and $\{ x_i \}_{i \in {\mathbb{N}}}$, where $x_i$ is a variable.  Take and fix a bijection $f : G_{< \infty}$ $\to$ $\{ x_i \}_{i \in {\mathbb{N}}}$ satisfying: the number of chords of $f^{-1}(x_i)$ is less than or equal to that of $f^{-1}(x_j)$ if and only if $i \le j$ $(i, j \in {\mathbb{N}})$.  For each positive integer $d$, let $G_{\le d}$ be the set of chord diagrams consisting of at most $d$ chords and let $n_d$ $=$ $|G_{\le d}|$.  Then, it is clear that $f|_{G_{\le d}}$ is a bijection from $G_{\le d}$ to $\{ x_1, x_2, \ldots, x_{n_d} \}$.  Further, for each pair of integers $b$ and $d$ $(2 \le b \le d)$, let $G_{b, d}$ $=$ $G_{\le d} \setminus G_{\le b-1}$.      
Then, $f|_{G_{b, d}}$ is a bijection $G_{b, d}$ $\to$ $\{ x_{n_{b-1} +1}, x_{n_{b-1} + 2}, \dots, x_{n_{d}} \}$.    
\end{notation}
In the rest of this paper, we use the notations in Notation~\ref{not2} unless otherwise denoted, and we freely use this identification between $G_{< \infty}$ and $\{x_i\}_{i \in \mathbb{N}}$.
\begin{definition}[a chord diagram $CD_P$ of a spherical/plane curve $P$]\label{dfn_cdp}
Let $P$ be a spherical/plane curve, i.e. there is a generic immersion $g: S^1 \to S^2$ or $\mathbb{R}^2$ such that $g(S^1)=P$.  
We define a chord diagram of $P$ (e.g., Figure~\ref{def1}) as follows: let $k$ be the number of the double points of $P$, and $m_1, m_2, \ldots, m_k$ mutually distinct positive integers.   
\begin{figure}[h!]
\includegraphics[width=5cm]{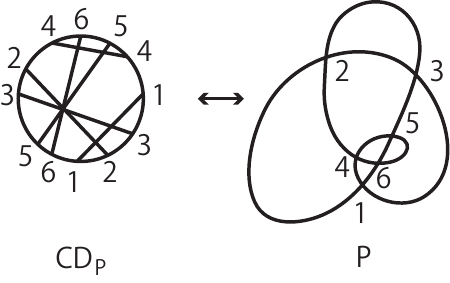}
\caption{A chord diagram $CD_P$ of a spherical/plane curve $P$.}\label{def1}
\end{figure} 
Fix a base point, which is not a double point on $P$, and choose an orientation of $P$.  Starting from the base point, proceed along $P$ according to the orientation of $P$.    
We assign $m_1$ to the first double point that we encounter.  Then we assign $m_2$ to the next double point that we encounter provided it is not the first double point.  Suppose that we have already assigned $m_1$, $m_2, \dots, m_p$.  Then we assign $m_{p+1}$ to the next double point that we encounter if it has not been assigned yet.    
Following the same procedure, we finally label all the double points.    
Note that $g^{-1}({\text{double point assigned $m_i$}})$ consists of two points on $S^1$ and we shall assign $m_i$ to them.  The chord diagram represented by $g^{-1}({\text{double point assigned $m_1$}}),$ $g^{-1}({\text{double point assigned $m_2$}}),$ $\dots,$ $g^{-1}({\text{double point assigned $m_k$}})$ on $S^1$ is denoted by $CD_P$ and is called {\it{a chord diagram of the spherical/plane curve}} $P$.  
\end{definition}
Recall that $CD_P$ gives an equivalence class of Gauss words, say $[v_P]$.  Then, by the definition of the equivalence relation, it is easy to see that the map $P \mapsto [v_P]$ is well-defined.  
\begin{notation}[$x(CD)$]\label{def_x(P)}
Let $x \in \{ x_i \}_{i \in \mathbb{N}}$ (hence, $x$ corresponds to a chord diagram).  For a given chord diagram $CD$, fix a Gauss word $G$ representing $CD$.    
Let $\sub_x (G)$ $=$ $\{H~|~H \in \sub(G)$, $[H]=x\}$.    
The cardinality of this subset is denoted by $x(G)$, i.e., $x(G)$ $=$ $|\sub_x (G)|$.    
Let $G'$ be another Gauss word representing $CD$.  
By the definition of the isomorphism of Gauss words, it is easy to see $x(G')=x(G)$.  Hence, we shall denote this number by $x(CD)$.     
If $CD$ is a chord diagram of a spherical/plane  curve $P$, then $x(CD)$ can be denoted by $x(P)$. 
\end{notation}
Since each equivalence class of Gauss words is identified with a chord diagram, we can calculate the number $x(CD)$ using geometric observations.  We explain this philosophy in the next example.  

\begin{example}\label{example1}
We consider the chord diagram $CD$ in Figure~\ref{def2} (Note that $CD$ $=$ $CD_P$ in Figure~\ref{def1}).  Then we label the chords of $CD$ by $\alpha_i$ $(1 \le i \le 6)$ as in Figure~\ref{def2}.   
Consider the subset of the power set of $\{ \alpha_1, \alpha_2, \dots, \alpha_6 \}$, each element of which represents a chord diagram 
isomorphic to $\otimes$.  It is elementary to see that this subset consists of ten elements, those are, $\{\alpha_1, \alpha_2\}$, $\{\alpha_1, \alpha_3\}$, $\{\alpha_1, \alpha_4\}$, $\{\alpha_2, \alpha_3\}$, $\{\alpha_2, \alpha_4\}$, $\{\alpha_3, \alpha_4\}$, $\{\alpha_1, \alpha_5\}$, $\{\alpha_2, \alpha_5\}$, $\{\alpha_3, \alpha_6\}$, and $\{\alpha_4, \alpha_6\}$, and that fact shows that $\x(CD)=10$.  Similarly, we have $\tr(CD)=6$ and $\h(CD)=8$.    
\begin{figure}[h!]
\includegraphics[width=3cm]{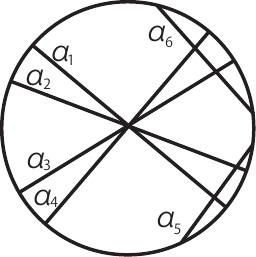}
\caption{$CD$.  
}\label{def2}
\end{figure}
\end{example}
Let $CD$ be a chord diagram consisting of chords $\alpha_1, \alpha_2, \dots, \alpha_k$.  Then the chord diagram consisting of a subset of $\{ \alpha_1, \alpha_2, \dots, \alpha_k \}$ is called a {\it{sub-chord diagram}} of $CD$. 
\begin{definition}[$\tilde{x}(z), \tilde{x}({[z]})$]\label{tilde_l}
Let $x$ be a chord diagram.  
We define the function $\tilde{x}$ from the set of Gauss words to $\{ 0, 1 \}$ by
 \[
{\tilde{x}} (F) = \begin{cases}
1 \quad [F] = x \\
0 \quad [F] \neq x.
\end{cases}
\]  
By definition, it is easy to see $\tilde{x}(F_1)$ $=$ $\tilde{x}(F_2)$ for each pair $F_1$ and $F_2$ with $[F_1]=[F_2]$.  Hence, we shall denote this number by $\tilde{x}([F_1])$.  If $[F]$ corresponds to a chord diagram of a spherical/plane curve $P$, then $\tilde{x}([F])$ is denoted by $\tilde{x}(P)$.  
Further, let $\mathbb{Z}[G_{\le l}]$ be the free $\mathbb{Z}$-module generated by the elements of $G_{\le l}$, where $l$ is sufficiently large.  
We linearly extend $\tilde{x}$ to a function from $\mathbb{Z}[G_{\le l}]$ to $\mathbb{Z}$.  
It is clear that for any Gauss word with $[G] = CD$, 
\begin{equation}\label{tilde_x}
x(CD) = \sum_{z \in \sub(G)} \tilde{x} (z).\end{equation}
\end{definition}

\subsection{Relators and Reidemeister moves}\label{rr1}
Let $\mathbb{Z}[G_{\le l}]$ be the free $\mathbb{Z}$-module defined in Subsection~\ref{section_def1}.     
In this subsection, first we define the elements of $\mathbb{Z}[G_{\le l}]$ called relators of types (I), (SI\!I), (WI\!I), (SI\!I\!I), and (WI\!I\!I).  
\begin{definition}[Relators, cf.~Figure~\ref{relator1a}]\label{def_relators}
\begin{itemize}
\item Type (I). An element $r$ of $\mathbb{Z}[G_{\le l}]$ is called a {\it{Type}}~(\ii)~{\it{relator}} if there exist a Gauss word $S$ and a letter $i$ not in $S$ such that $r$ $=$ $[Sii]$.     
\item Type (SI\!I).   
An element $r$ of $\mathbb{Z}[G_{\le l}]$ is called a {\it{Type}}~(\s)~{\it{relator}} if there exist a Gauss word $ST$ and letters $i$ and $j$ not in $ST$ such that $r$ $=$ $  
[SijTji]$ $+$ $[SiTi]$ $+$ $[SjTj]$.    
\item Type (WI\!I).  
An element $r$ of $\mathbb{Z}[G_{\le l}]$ is called a {\it{Type}}~(\w)~{\it{relator}} if there exists a Gauss word $ST$ and letters $i$ and $j$ not in $ST$ such that $r$ $=$ $
[SijTij]$ $+$ $[SiTi]$ $+$ $[SjTj]$.  
\item Type (SI\!I\!I).  
An element $r$ of $\mathbb{Z}[G_{\le l}]$ is called a {\it{Type}}~(\sss)~{\it{relator}} if there exists a Gauss word $STU$ and letters $i$, $j$, $k$ not in $STU$ such that $r$ $=$ $([SijTkiUjk]$ $+$ $[SijTiUj]$ $+$ $[SiTkiUk]$ $+$ $[SjTkUjk])$ $-$ $([SjiTikUkj]$ $+$ $[SjiTiUj]$ $+$ $[SiTikUk]$ $+$ $[SjTkUkj])$.  
\item Type (WI\!I\!I).  
An element $r$ of $\mathbb{Z}[G_{\le l}]$ is called a {\it{Type}}~(\www)~{\it{relator}} if there exists a Gauss word $STU$ and letters $i$, $j$, $k$ not in $STU$ such that $r$ $=$ $( [SijTikUjk]$ $+$ $[SijTiUj]$ $+$ $[SiTikUk]$ $+$ $[SjTkUjk] )$ $-$  $( [SjiTkiUkj]$ $+$ $[SjiTiUj]$ $+$ $[SiTkiUk]$ $+$ $[SjTkUkj] )$.
\end{itemize} 
\end{definition}
\begin{figure}
\includegraphics[width=8cm]{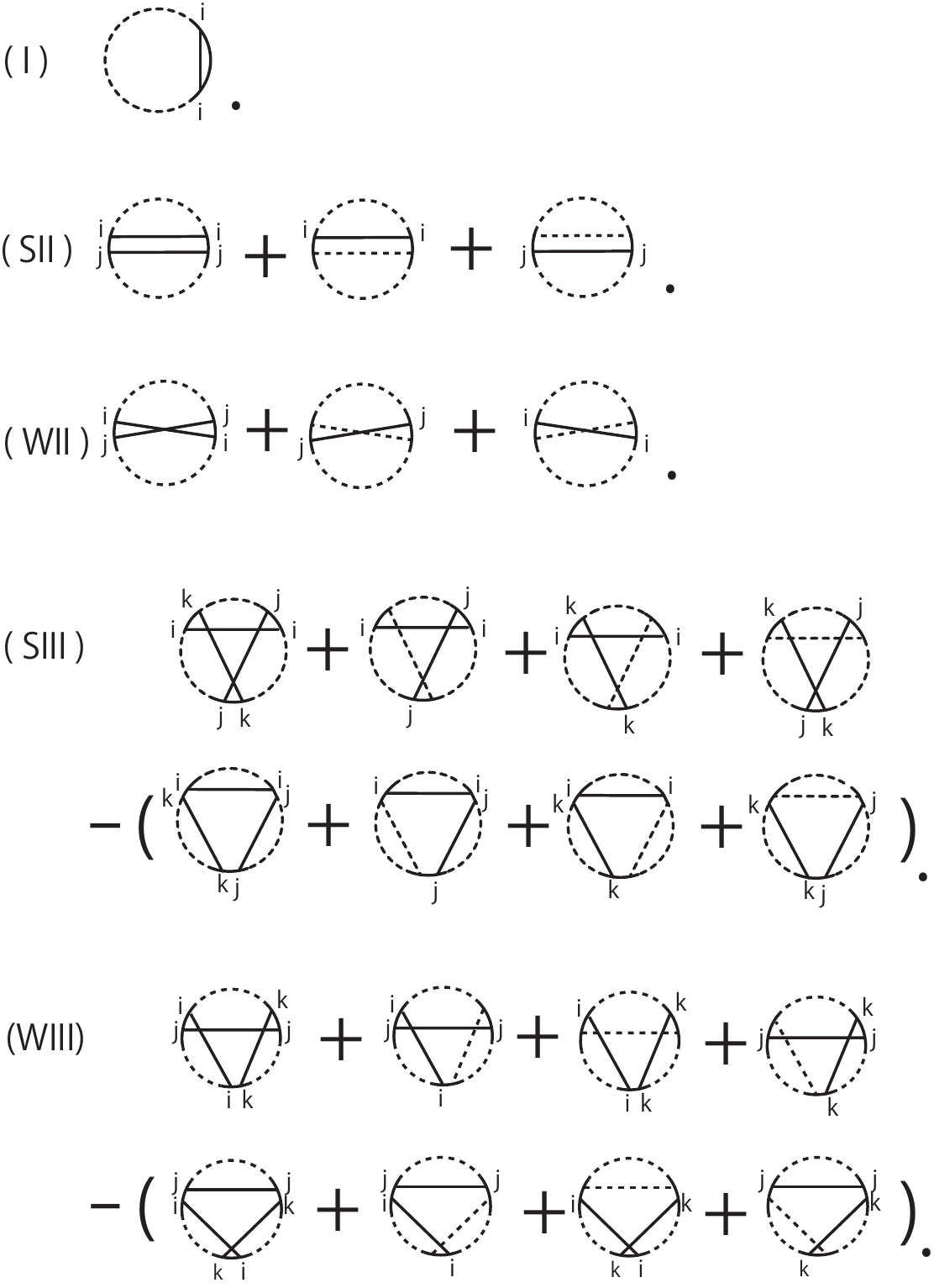}
\caption{Relators.}\label{relator1a}
\end{figure}

We note that in Definition~\ref{def_relators}, each relator is induced by the corresponding Reidemeister move: Type (I) relator corresponding to RI, Type (SI\!I) relator corresponding to strong~RI\!I, Type (WI\!I) relator corresponding to weak~RI\!I, Type (SI\!I\!I) relator corresponding to strong~RI\!I\!I, and Type (WI\!I\!I) relator corresponding to weak~RI\!I\!I.  For the precise statement of this note, we introduce the following setting.  
\begin{figure}[h!]
\includegraphics[width=12cm]{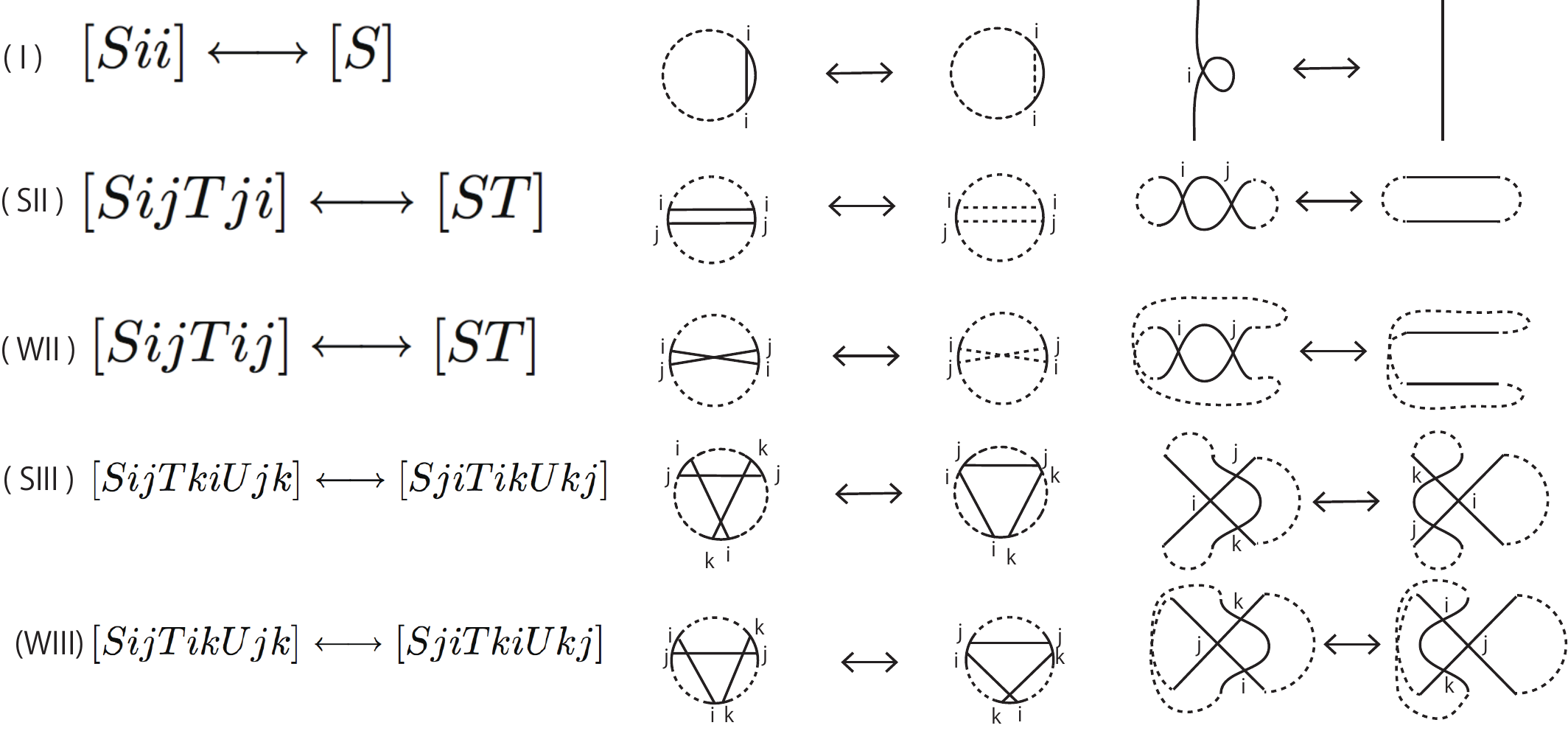}
\caption{Reidemeister moves.}\label{def3}
\end{figure}

We first introduce the following notation.    
Let $P$ and $P'$ be two spherical/plane curves.  
If $P$ and $P'$ are related by a single RI (strong~RI\!I, weak~RI\!I, strong~RI\!I\!I, or weak~RI\!I\!I resp.), then there are Gauss words $G$ and $G'$ such that (by exchanging $P$ and $P'$, if necessary), $G$ $=$ $Sii$ ($SijTji$, $SijTij$, $SijTkiUjk$, or $SijTikUjk$ resp.) and $G'=S$ ($ST$, $ST$, $SjiTikUkj$, or $SjiTkiUkj$ resp.) such that $[G]=CD_P$ and $[G']=CD_{P'}$ (Figure~\ref{def3}).    The subset of $\sub(G)$ such that each element has exactly $m$ letters of $i$, $j$, and $k$ is denoted by $\sub^{(m)} (G)$.    By definition, 
\begin{equation}\label{sub}
\sub(G) = \sub^{(0)}(G) \amalg \sub^{(1)} (G) \amalg \sub^{(2)} (G) \amalg \sub^{(3)} (G).   
\end{equation}
Similarly, for a chord diagram $x$, $\sub^{(m)}_x (G)$ denotes the subset of $\sub_x(G)$ consisting of elements, each of which has exactly $m$ letters of $i$, $j$, and $k$.  Then,  
\begin{equation}\label{subx}
\sub_x (G) = \sub^{(0)}_x (G) \amalg \sub^{(1)}_x (G) \amalg \sub^{(2)}_x (G) \amalg \sub^{(3)}_x (G).  
\end{equation}

Let $\mathcal{C}$ be the set of ambient isotopy classes of spherical/plane curves.  
Next, for each element $\sum_i \alpha_i x_i \in \mathbb{Z}[G_{\le l}]$, we define a function $\mathcal{C}$ $\to$ $\mathbb{Z}$, also denoted by $\sum_i \alpha_i x_i$, and another function denoted by $\sum_i \alpha_i \tilde{x}_i$.      
\begin{definition}[$\sum_i \alpha_i x_i$, $\sum_i \alpha_i \tilde{x}_i$]\label{def_iv}
Let $b$ and $d$ ($2 \le b \le d$) be integers and $G_{\le d}$, $G_{b, d}$, $\{ x_i \}_{i \in \mathbb{N}}$, and $\mathbb{Z}[G_{\le l}]$ be as in Subsection~\ref{section_def1}.     
Recall that $n_d$ $=$ $|G_{\le d}|$ and $G_{b, d}$ $=$ $\{ x_i \}_{n_{b-1} +1 \le i \le n_{d}}$ and each $x_i$ represents the chord diagram $f^{-1}(x_i)$.  For each element
\[ \sum_{ n_{b-1} +1 \le i \le n_{d} } \alpha_i x_i \in \mathbb{Z}[G_{\le l}], \]
we define an integer-valued function $\mathcal{C}$ $\to$ $\mathbb{Z}$, also denoted by $\Fsum$, by
\[P \mapsto \Fsum (P),\]
where $x_i (P)$ is the integer introduced in Notation~\ref{def_x(P)}.   

Analogously, for each $\Fsum \in \mathbb{Z}[G_{\le l}]$, we define the function
\[
\Nsum \at : \mathbb{Z}[G_{\le l}] \to \mathbb{Z}
\]
by
\[
\left( \Nsum \at \right)(P) = \Nsum \at(P),   
\]
where $\tilde{x}_i (P)$ is the integer introduced in Definition~\ref{tilde_l}.
\end{definition}

Recall that if a spherical/plane curve $P$ and another spherical/plane curve $P'$ are related by a single RI, then there exist a letter $i$ and a Gauss word $S$ such that $CD_{P}$ $=$ $[Sii]$ and $CD_{P'}$ $=$ $[S]$.  
Note that in this case in the decomposition (\ref{sub}), $\sub^{(2)}(G)$ $=$ $\emptyset$ and $\sub^{(3)}(G)$ $=$ $\emptyset$, where $G=Sii$.  
Then, by (\ref{tilde_x}) in Definition~\ref{tilde_l} and (\ref{sub}), 
\begin{equation*}
\begin{split}
\Fsum &(P) = \sFsum \left( \sum_{z \in \sub(G)} \tilde{x}_i (z) \right) \\
\!\!\!&= \sFsum \left( \sum_{z_0 \in \sub^{(0)}(G)} \tilde{x}_i (z_0) + \sum_{z_1 \in \sub^{(1)}(G)} \tilde{x}_i (z_1) \right).    
\end{split}
\end{equation*}
Note that each element $z_0 \in$ $\sub^{(0)}(G)$ is a sub-Gauss word of $S$.  Then it is clear that $\sub^{(1)}(G)$ $=$ $\{ z_0 ii~|~z_0 \in \sub^{(0)}(G) \}$.  Hence, 
\begin{align*}
\Fsum (P) 
&= \sFsum \left( \sum_{z_0 \in \sub^{(0)}(G)} \tilde{x}_i (z_0) + \sum_{z_0 \in \sub^{(0)}(G)} \tilde{x}_i (z_0 ii) \right) \\
&= \sFsum \left( \sum_{z_0 \in \sub^{(0)}(G)} \tilde{x}_i ([z_0]) + \sum_{z_0 \in \sub^{(0)}(G)} \tilde{x}_i ([z_0 ii]) \right).       
\end{align*}
On the other hand, since $\sub(G')$ is identified with $\sub^{(0)}(G)$,   
\begin{align*}
\Fsum(P') &= \sFsum \left(\sum_{z' \in \sub (G')} \tilde{x}_i (z') \right) (\because (\ref{tilde_x})) \\
&=\sFsum \left(\sum_{z_0 \in \sub^{(0)} (G)} \tilde{x}_i (z_0) \right)\\
&= \sFsum \left( \sum_{z_0 \in \sub^{(0)} (G)} \tilde{x}_i \big( [z_0] \big) \right).  
\end{align*}
As a conclusion, the difference of the values is calculated as follows. 
\begin{align*}
\Fsum(P) &-\Fsum(P')\\
&= \sum_{n_{b-1} + 1 \le i \le n_d} \sum_{z_0 \in \sub^{(0)}(G)} \alpha_i \tilde{x}_i \left( [z_0ii] \right).
\end{align*}
We note that this is a linear combination of the values of Type~(\ii) relators via $\tilde{x}_i$.   
  
For the cases of Type (\s), (\w), (\sss), or (\www) relators, the arguments are slightly more complicated than that of Type (\ii) relator.  We will explain them in the proof of Theorem~\ref{gg_thm1} and thus, we omit them here.       
\begin{definition}[$R_{\epsilon_1 \epsilon_2 \epsilon_3 \epsilon_4 \epsilon_5}$]\label{dfn_relator1}
For each $(\epsilon_1, \epsilon_2, \epsilon_3, \epsilon_4, \epsilon_5) \in \{0, 1\}^{5}$, let $R_{\epsilon_1 \epsilon_2 \epsilon_3 \epsilon_4 \epsilon_5}$ $=$ $\cup_{\epsilon_i = 1} R_i$ ($\subset \cup_{l \ge 1} \mathbb{Z}[G_{\le l}]$), where $R_1$ is the set of Type (I) relators, $R_2$ is the set of Type (SI\!I) relators, $R_3$ is the set of Type (WI\!I) relators, $R_4$ is the set of Type (SI\!I\!I) relators, and $R_5$ is the set of Type (WI\!I\!I) relators.     
\end{definition}

For integers $b$ and $d$ ($2 \le b \le d$), let $O_{b, d}$ be the projection $\mathbb{Z}[G_{\le l}]$ $\to$ $\mathbb{Z}[G_{b, d}]$.  Here, note that $O_{b, d}$ is a linear map.  By the definition, we immediately have:
\begin{lemma}\label{lem_relator}
If $\ineq$, then for any $r \in \mathbb{Z}[G_{\le l}]$, 
\[
\tilde{x}_i (r) = \tilde{x}_i (O_{b, d} (r)).  
\]
\end{lemma}
\begin{notation}\label{orep1}
Let $R_{\epsilon_1 \epsilon_2 \epsilon_3 \epsilon_4 \epsilon_5}(b, d)$ $=$ $O_{b, d}(R_{\epsilon_1 \epsilon_2 \epsilon_3 \epsilon_4 \epsilon_5})$. 
\end{notation}
By using Lemma~\ref{lem_relator}, we have the next proposition.   
\begin{proposition}\label{relator_prop}
For each pair of integers $b$ and $d$ $(2 \le b \le d)$, let $\Nsum \at$ be a function as in Definition~\ref{def_iv}.  
For $(\epsilon_1, \epsilon_2, \epsilon_3, \epsilon_4, \epsilon_5)$ $\in \{ 0, 1 \}^5$, let $\Rep$ be the set as in Definition~\ref{dfn_relator1}.
The following two statements are equivalent: 
\begin{enumerate}
\item $\Nsum \alpha_i \tilde{x}_i (r) = 0 \quad (\forall r \in \Rep)$.\label{s1}
\item $\Nsum \alpha_i \tilde{x}_i (r) = 0 \quad (\forall r \in R_{\epsilon_1 \epsilon_2 \epsilon_3 \epsilon_4 \epsilon_5}(b, d))$.\label{s2}
\end{enumerate}
\end{proposition}
\begin{proof}
\begin{itemize}
\item $(\ref{s1}) \Rightarrow (\ref{s2})$.  This is obvious.  
\item $(\ref{s1}) \Leftarrow (\ref{s2})$. Let $r \in \Rep$.  By Lemma~\ref{lem_relator},  
\[
\Nsum \at (r) = \Nsum \at (O_{b, d}(r)).  
\]
By condition (\ref{s2}), 
\[
\Nsum \at (O_{b, d}(r)) = 0.  
\]
Then, 
\[
\Nsum \at (r) = 0.  
\]
\end{itemize}
\end{proof}
We note that it is a simple task to give an explicit presentation of $\Rep(b, d)$.  In particular, if $d \le 4$, one can do it by hands.  The next example shows the concrete such process for a certain case.     
\begin{example}[$R_{00010}(2, 3)$]\label{ex1_relator}
$R_{00010}(2, 3)$ $=$ $\{ \tr + 3 \x - 3 \chordth - \tri, 3 \chordtwo - 2 \thc - \tri, \tr + \h - 2 \chordtwo \}$.  This fact is confirmed as follows.  

Recall that a relator of Type~(\sss) is of the form $r$ $=$ $([SijTkiUjk]$ $+$ $[SijTiUj]$ $+$ $[SiTkiUk]$ $+$ $[SjTkUjk])$ $-$ $([SjiTikUkj]$ $+$ $[SjiTiUj]$ $+$ $[SiTikUk]$ $+$ $[SjTkUkj])$.    
We immediately see, from this expression, that each term of $r$ consists of $n$ or $n+1$ letters for some $n$.  
Note that $O_{2, 3}(r)$ is obtained from $r$ by removing the terms of length $< 2$, or $> 3$. Hence it is enough to suppose that each term of $r$ consists of two or three letters (Case~1) or each term of $r$ consists of three or four letters (Case~2).      

\noindent$\bullet$ Case~1. Each term of $r$ consists of two or three letters.  

In this case, we have $STU$ $=$ $\emptyset$ (i.e., $S=\emptyset$, $T=\emptyset$, and $U=\emptyset$), and, hence, 
\begin{align*}
O_{2, 3}(r) &= r = [ijkijk] + [ijij] + [ikik] + [jkjk] - [jiikkj] - [jiij] - [iikk] - [jkkj] \\ 
&= [ijkijk] + 3[ijij] - 3 [iijj] - [iijjkk].
\end{align*}

\noindent $\bullet$ Case~2. Each term of $r$ consists of three or four letters.  

In this case,  
$STU$ $=$ $\sigma \sigma$
 for some letter $\sigma$ ($\neq i, j, k$).  

Case~2.1. Both $\sigma$'s are contained in one of the sub-words $S$, $T$, and $U$.

Suppose that both $\sigma$'s are contained in $S$.   Then, 
\begin{align*}
r&= [\sigma \sigma ijkijk] + [\sigma \sigma ijij] + [\sigma \sigma ikik] + [\sigma \sigma jkjk] - [\sigma \sigma jiikkj] - [\sigma \sigma jiij] \\&\qquad - [\sigma\sigma iikk] - [\sigma\sigma jkkj]\\
&= [\sigma \sigma ijkijk] - [\sigma \sigma jiikkj] + 3[ijijkk] - 2[ijkkji] - [iijjkk].
\end{align*}
Hence, 
\[
O_{2, 3}(r) = 3[ijijkk] - 2[ijkkji] - [iijjkk].
\]
By using similar arguments, we can show that we obtain the same element as above in other cases.  Details of the calculations are left to the reader.

Case~2.2.  The two $\sigma$'s are contained in mutually different sub-words $S$, $T$, $U$.

Suppose that one $\sigma$ is contained in $S$ and the other $\sigma$ is contained in $T$.  Then  
\begin{align*}
r &= [\sigma ij\sigma kijk] + [\sigma ij \sigma ij] + [\sigma i \sigma kik] + [\sigma j \sigma kjk] - [\sigma ji\sigma ikkj] - [\sigma ji\sigma ij] \\& \qquad - [\sigma i\sigma ikk] - [\sigma j\sigma kkj]. 
\end{align*}
Here, we note that
\[
[\sigma i\sigma ikk] = [\sigma j\sigma kkj] = [ijijkk]
\]
and
\[
[\sigma i \sigma kik] = [\sigma j \sigma kjk] = [\sigma ji\sigma ij] = [ijkikj].
\]
These show that
\begin{align*}
r&= [\sigma ij \sigma kijk] - [\sigma ji \sigma ikkj] + [ijkijk] + [ijkikj] - 2 [ijijkk].
\end{align*}
Hence, 
\[
O_{2, 3}(r) = [ijkijk] + [ijkikj] - 2 [ijijkk].
\]
By using similar arguments, we can show that we obtain the same element as above in other cases.  Details of the calculations are left to the reader.     

As a conclusion, 
$R_{00010} (2, 3)$ $=$ $\{ [ijkijk] + 3[ijij] - 3 [iijj] - [jiikkj],$ $3[ijijkk] - 2[ijkkji] - [iijjkk]$, $[ijkijk] + [ijkikj] - 2 [ijijkk] \}$ $=$ $\{ \tr + 3 \x - 3 \chordth - \tri, 3 \chordtwo - 2 \thc - \tri, \tr + \h - 2 \chordtwo \}$.  
\end{example}
\begin{example}[$R_{00010} (2, 4)$]\label{ex1_relator34}
By using the arguments in Example~\ref{ex1_relator}, we can show: 
$
R_{00010}(2, 4)$ $
=$ $\{ \tr + 3 \x - 3 \chordth - \tri,$ $\triso + 3 \chordtwo - \chordthab - 2 \thc - \tri,$ 
$\ma + \tr + \h -\trii - 2 \chordtwo, 2 \htri - \crh - \chordtwoc,$
  $\ma + \hthc - \sh - 2 \triii,$  
  $\qua + 2 \tri - \ma - 2 \triso,$  $2 \ma + \sh - 2 \htri - \crh,$  $\htri + 2 \crh -  \hthc - 2 \trii,$ $3 \chordtwoc - 2 \chordtha - \chordthb,$ $\chordtha - 2 \four - \chordthab,$ $\chordthb - 2 \chordthab - \triFour,$ $\trisox + \chordtwx - \four - 2 \chordthab
 \}$. 
Details of the proof are left to the reader.  
\end{example}
\begin{example}[$R_{00001} (2, 3)$]\label{ex_00001}
By using the arguments in Example~\ref{ex1_relator}, we can show: 
\[
R_{00001}(2, 3) = \{ \h + \x - \chordth - \chordtwo, \tr - \h, 3 \h - \tr - 2 \chordtwo,
\chordtwo - 2 \thc + \tri, \chordtwo - \tri
 \}. 
\]
Details of the proof are left to the reader.  
\end{example}

\subsection{Main result and corollaries}\label{sec_main_result_1}
\begin{theorem}\label{gg_thm1}
Let $b$ and $d$ $(2 \le b \le d)$ be integers and $G_{\le d}$, $\{ x_i \}_{i \in \mathbb{N}}$, $\mathbb{Z}[G_{\le l}]$, $n_d$ $=$ $|G_{\le d}|$, and $G_{b, d}$ $=$ $\{ x_i \}_{n_{b-1} +1 \le i \le n_{d}}$ be as in Subsection~\ref{section_def1}.     
Let $\Fsum$ and $\aFsum$ be functions as in Definition~\ref{def_iv}.  
For $(\epsilon_1, \epsilon_2, \epsilon_3, \epsilon_4, \epsilon_5)$ $\in \{ 0, 1 \}^5$, let $\Rep(b, d)$ be as in Subsection~\ref{rr1}.  
Suppose the following conditions are satisfied: 

\noindent $\bullet$ If $\epsilon_1 =1$, $\displaystyle \aFsum(r)=0$ for each $r \in R_{10000}(b, d)$.  

\noindent $\bullet$ If $\epsilon_2 =1$, $\displaystyle \aFsum(r)=0$ for each $r \in R_{01000}(b, d)$.

\noindent $\bullet$ If $\epsilon_3 =1$, $\displaystyle \aFsum(r)=0$ for each $r \in R_{00100}(b, d)$.

\noindent $\bullet$ If $\epsilon_4 =1$, $\displaystyle \aFsum(r)=0$ for each $r \in R_{00010}(b, d)$.  

\noindent $\bullet$ If $\epsilon_5 =1$, $\displaystyle \aFsum(r)=0$ for each $r \in R_{00001}(b, d)$.   

Then, the function $\displaystyle \Fsum$ on the set of ambient isotopy classes of spherical/plane  curves is invariant under the Reidemeister moves corresponding to $\epsilon_j =1$.   
\end{theorem}
\begin{definition}[irreducible chord diagram]
Let $x$ be a chord diagram.  A chord $\alpha$ in $x$ is said to be an \emph{isolated chord} if $\alpha$ does not intersect any other chord.  If $x$ has an isolated chord, $x$ is called \emph{reducible} and otherwise, $x$ is called \emph{irreducible}.      
The set of the irreducible chord diagrams is denoted by $\irr$.  Let $\irri$ $=$ $\{ i ~|~ \ineq, x_i \in \irr \}$.
\end{definition}
If we consider the function of the form $\FsumIrr$ for $\Fsum$ in Theorem~\ref{gg_thm1}, we obtain:  
\begin{corollary}\label{g_thm1}
Let $b$ and $d$ $(2 \le b \le d)$ be integers and $G_{\le d}$, $\{ x_i \}_{i \in \mathbb{N}}$, $\mathbb{Z}[G_{\le l}]$, $n_d$ $=$ $|G_{\le d}|$, and $G_{b, d}$ $=$ $\{ x_i \}_{n_{b-1} +1 \le i \le n_{d}}$ be as in Subsection~\ref{section_def1}.     
Let $\FsumIrr$ and $\aFsumIrr$ be functions as in Definition~\ref{def_iv}.  
For $(\epsilon_1, \epsilon_2, \epsilon_3, \epsilon_4, \epsilon_5)$ $\in \{ 0, 1 \}^5$, let $\Rep(b, d)$ be as in Subsection~\ref{rr1}.  
Suppose the following conditions are satisfied: 

\noindent $\bullet$ If $\epsilon_2 =1$, $\displaystyle \aFsumIrr(r)=0$ for each $r \in R_{01000}(b, d)$.

\noindent $\bullet$ If $\epsilon_3 =1$, $\displaystyle \aFsumIrr(r)=0$ for each $r \in R_{00100}(b, d)$.

\noindent $\bullet$ If $\epsilon_4 =1$, $\displaystyle \aFsumIrr(r)=0$ for each $r \in R_{00010}(b, d)$.

\noindent $\bullet$ If $\epsilon_5 =1$, $\displaystyle \aFsumIrr(r)=0$ for each $r \in R_{00001}(b, d)$.   

Then, the function
$\displaystyle \FsumIrr$
 on the set of ambient isotopy classes of spherical/plane curves is invariant under RI and the Reidemeister moves corresponding to $\epsilon_j =1$.    
\end{corollary}
\noindent{\bf{Proof of Corollary~\ref{g_thm1} from Theorem~\ref{gg_thm1}.}}  
By Theorem~\ref{gg_thm1}, it is enough to show
\begin{equation}\label{eq2}
\aFsumIrr (r)=0  \quad (\forall r \in R_{10000} (b, d))
\end{equation}
for a proof of Corollary~\ref{g_thm1}.  We first note that if $x_i \in \irr$, then $x_i$ has no isolated chords.  On the other hand, let $r \in R_{10000}(b, d)$, that is, there exist a Gauss word $S$ and a letter $j$ such that $r=[Sjj]$.  Then, the chord corresponding to $j$ is isolated.  
These show that $\tilde{x}_i (r)=0$.  This shows that (\ref{eq2}) holds.         

Note that $\FsumIrr$ is a special case of $\Njsum \alpha_j x_j$.  Thus, we may suppose that integers $\alpha_j$ ($n_{b-1} + 1 \le j \le n_{d}$) satisfy that $\FsumIrr$ $=$  $\Njsum \alpha_j x_j$.  Therefore, if $\aFsumIrr (r)$ $=0$, then $\Njsum \alpha_j \tilde{x}_j (r)$ $=0$.  By Theorem~\ref{gg_thm1}, $\sum_{n_{b-1} + 1 \le j \le n_{d}} \alpha_j x_j$ is invariant.  Since each non-zero term is common, $\FsumIrr$ is also invariant.  
$\hfill \Box$ 
\begin{example}\label{ex_00010}
It is easy to see that $\{ x_i \}_{i \in \irritwo}$ consists of three elements, $y_1$ $=$ $\x$, $y_2$ $=$ $\tr$, and $y_3$ $=$ $\h$.  This fact together with Corollary~\ref{g_thm1} implies that if there exist integers $\alpha_1$, $\alpha_2$, and $\alpha_3$ such that $(\alpha_1 \tilde{y}_1$ $+$ $\alpha_2 \tilde{y}_2$ $+$ $\alpha_3 \tilde{y}_3)(r)$ $=$ $0$ ($\forall r \in {R}_{00010}(2, 3)$), then $\alpha_1 y_1$ $+$ $\alpha_2 y_2$ $+$ $\alpha_3 y_3$ is invariant under RI and strong RI\!I\!I. 
Recall that $R_{00010}(2, 3)$ consists of the following three elements $r_1$, $r_2$, and $r_3$ (Example~\ref{ex1_relator}):  
\begin{align*}
r_1 &= \tr + 3 \x - 3 \chordth - \tri, \\
r_2 &= 3 \chordtwo - 2 \thc - \tri,~{\text{and}}~ \\
r_3 &= \tr + \h - 2 \chordtwo.  
\end{align*}
\begin{figure}[b]
\includegraphics[width=12cm]{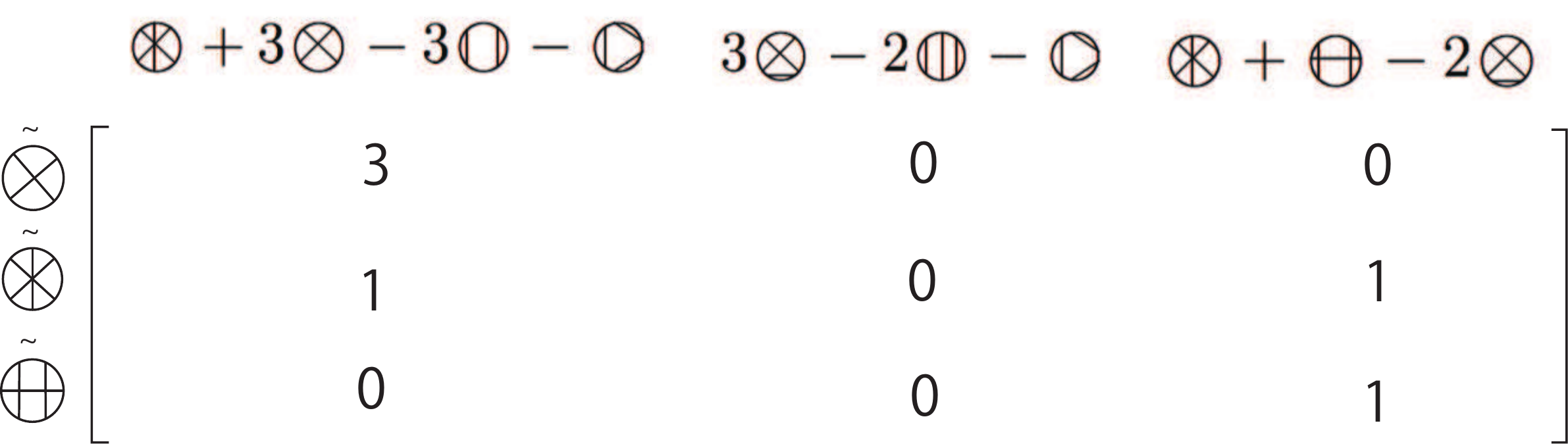}
\caption{Matrix $(\tilde{y}_i (r_j))_{1 \le i, j \le 3}.$}\label{matrix1b}
\end{figure}
We can find such $\alpha_1$, $\alpha_2$, and $\alpha_3$ by solving the linear equation ${\bf{x}} M$ $=$ $\bf{0}$, where $M$ $=$ $(\tilde{y}_i (r_j))$ (Figure~\ref{matrix1b}).  It is elementary to show that the set of solutions is $\{ {\bf{x}}$ $=$ $\gamma (1, -3, 3)$ $|~\gamma \in \mathbb{Z}\}$.  
Particularly, $y_1$ $-$ $3 y_2$ $+$ $3 y_3$ is invariant under RI and strong RI\!I\!I.  We note that $\frac{1}{4} (y_1$ $-$ $3 y_2$ $+$ $3 y_3)$ is the function $\lambda$ introduced in \cite{IT3} from a different view point.           
\end{example}
\begin{example}
In contrast to Example~\ref{ex_00010}, we will note that there does not exist a nonzero $(\alpha_1, \alpha_2, \alpha_3)$ such that $(\alpha_1 \tilde{y}_1$ $+$ $\alpha_2 \tilde{y}_2$ $+$ $\alpha_3 \tilde{y}_3)(r)$ $=$ $0$ for each $r \in R_{00001}(2, 3)$, i.e., there does not exist $\alpha_1 \tilde{y_1}$ $+$ $\alpha_2 \tilde{y}_2$ $+$ $\alpha_3 \tilde{y}_3$ satisfying the condition of Corollary~\ref{g_thm1} with $\epsilon_5 =1$.    
In fact, we take $y_i (i = 1, 2, 3)$ as in Example~\ref{ex_00010}, and name the elements of $R_{00001}(2, 3)$ (see Example~\ref{ex_00001}) as: 
\begin{align*}
r'_1 &= \h + \x - \chordth - \chordtwo, \\
r'_2 &= \tr - \h, \\
r'_3 &= 3 \h - \tr - 2 \chordtwo, \\
r'_4 &= \chordtwo - 2 \thc + \tri,~{\text{and}}~\\
r'_5 &= \chordtwo - \tri.\\
\end{align*}
Let $M$ $=$ $(\tilde{y}_i (r'_j))_{1 \le i \le 3, 1 \le j \le 5}$ (Figure~\ref{matrix3}).  
Then it is easy to see that the linear equation  ${\bf{x}} M$ $=$ $\bf{0}$ admits the trivial solution ${\bf{x}}$ $=$ $(0, 0, 0)$ only.         
\begin{figure}[h!]
\includegraphics[width=12cm]{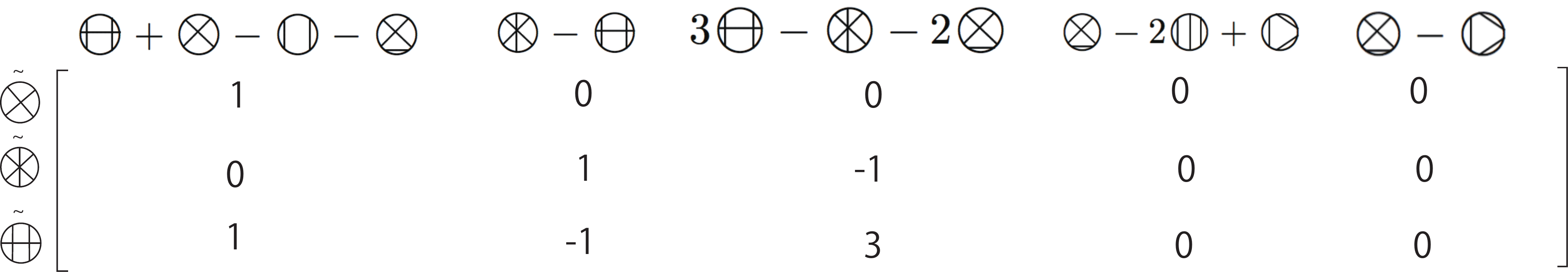}
\caption{Matrix $(\tilde{y}_i (r'_j))_{1 \le i \le 3, 1 \le j \le 5}.$}\label{matrix3}
\end{figure}  
\end{example}
\begin{definition}[connected chord diagram]
Let $v$ be a Gauss word of length $2m$ and $w$ a Gauss word of length $2n$ where $v(\hat{2m})$ $\cap$ $w(\hat{2n})$ $=$ $\emptyset$.    
Then we define the Gauss word of length $2(m+n)$, denoted by $vw$, by $vw(i)=v(i)$ ($1 \le i \le 2m$) and $vw(2m+i)=w(i)$ ($1 \le i \le 2n$).  The chord diagram $[vw]$ is called a {\it{product}} of chord diagrams $[v]$ and $[w]$.  If a chord diagram is not a product of two non-empty chord diagrams, then the chord diagram is called a {\it{connected}} chord diagram.  The set of the connected chord diagrams is denoted by $\conn$.  Let $\conni$ $=$ $\{ i ~|~ \ineq, x_i \in \conn \}$.  
\end{definition}
It is easy to see that if a connected chord diagram has at least two chords, it is irreducible.  
Note that the chord diagram consisting of exactly one chord is connected but not irreducible.  

Recall that $\mathcal{C}$ is the set of ambient isotopy classes of spherical/plane curves.   In Definition~\ref{additivity_def}, for plane curves, we freely use the identification of $S^2$ with $\mathbb{R}^2 \cup \{ \infty \}$.   
\begin{definition}[connected sum]\label{additivity_def}
Let $P_1$, $P_2$ $\in \mathcal{C}$.  We suppose that the ambient $2$-spheres are oriented.  
Let $p_i$ be a point on $P_i$ such that $p_i$ is not a double point ($i=1, 2$).  Let $d_i$ be a sufficiently small disk with center $p_i$ ($i=1, 2$) such that $d_i \cap P_i$ consists of an arc properly embedded in $d_i$.  Let $\hat{d}_i$ $=$ $cl(S^2 \setminus d_i)$ and $\hat{P}_i$ $=$ $P_i \cap \hat{d}_i$; let $h : \partial \hat{d}_1$ $\to$ $\partial \hat{d}_2$ be an orientation reversing homeomorphism such that $h(\partial \hat{P}_1)$ $=$ $\partial \hat{P}_2$.  Then, $\hat{P}_1 \cup_h \hat{P}_2$ gives a spherical curve in the oriented $2$-sphere $\hat{d}_1 \cup_h \hat{d}_2$.  The spherical curve $\hat{P}_1 \cup_h \hat{P}_2$ in the oriented $2$-sphere, denoted by $P_1 \sharp_{(p_1,~p_2), h} P_2$, is called a {\it{connected sum}} of the spherical/plane curves $P_1$ and $P_2$ at the pair of points $p_1$ and $p_2$ (see Figure~\ref{connect}).   
\begin{figure}[h!]
\includegraphics[width=10cm]{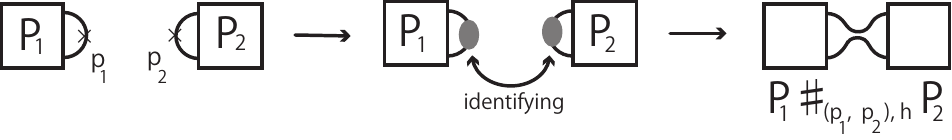}
\caption{Connected sum $P_1 \sharp_{(p_1,~p_2), h} P_2$ of two spherical/plane curves $P_1$ and $P_2$.}\label{connect}
\end{figure}
\end{definition}
\begin{definition}[additivity]   
Let $I$ be a $\mathbb{Z}$-valued function on $\mathcal{C}$.  We say that $I$ is {\it{additive}} if $I(P_1 \sharp_{(p_1, p_2), h} P_2)$ $=$ $I(P_1) + I(P_2)$ for any spherical/plane curves $P_1$ and $P_2$ and for any $P_1 \sharp_{(p_1, p_2), h} P_2$ in Definition~\ref{additivity_def}.  
\end{definition}
\begin{definition}[prime]
A spherical/plane curve $P$ is \emph{trivial} if $P$ is a simple closed curve.  
A spherical/plane curve $P$ is {\it{prime}} if $P$ cannot be expressed as a connected sum of two nontrivial spherical/plane curves.  
\end{definition}

If we consider the function of the form $\FsumConn$ for $\Fsum$ in Theorem~\ref{gg_thm1}, we obtain:   
\begin{corollary}\label{cor3}
Let $b$ and $d$ $(2 \le b \le d)$ be integers and $G_{\le d}$, $\{ x_i \}_{i \in \mathbb{N}}$, $\mathbb{Z}[G_{\le l}]$, $n_d$ $=$ $|G_{\le d}|$, and $G_{b, d}$ $=$ $\{ x_i \}_{n_{b-1} +1 \le i \le n_{d}}$ be as in Subsection~\ref{section_def1}.     
Let $\FsumConn$ and $\aFsumConn$ be functions as in Definition~\ref{def_iv}.  
For $(\epsilon_1, \epsilon_2, \epsilon_3, \epsilon_4, \epsilon_5)$ $\in \{ 0, 1 \}^5$, let $\Rep(b, d)$ be as in Subsection~\ref{rr1}.  
Suppose the following conditions are satisfied:

\noindent $\bullet$ If $\epsilon_2 =1$, $\displaystyle \aFsumConn(r)=0$ for each $r \in R_{01000}(b, d)$.

\noindent $\bullet$ If $\epsilon_3 =1$, $\displaystyle \aFsumConn(r)=0$ for each $r \in R_{00100}(b, d)$.

\noindent $\bullet$ If $\epsilon_4 =1$, $\displaystyle \aFsumConn(r)=0$ for each $r \in R_{00010}(b, d)$.

\noindent $\bullet$ If $\epsilon_5 =1$, $\displaystyle \aFsumConn(r)=0$ for each $r \in R_{00001}(b, d)$.   

Then, the function
$\displaystyle \FsumConn$
 on the set of ambient isotopy classes of spherical/plane curves is additive and invariant under RI and the Reidemeister moves corresponding to $\epsilon_j =1$.      
\end{corollary}
\noindent{\bf{Proof of Corollary~\ref{cor3} from Theorem~\ref{gg_thm1}.}}
Since $b \ge 2$ and $i \ge n_{b-1} + 1$, each $x_i$ consists of more than one chords, i.e., $x_i \in \irr$ (see the note preceding Definition~\ref{additivity_def}).  

By Theorem~\ref{gg_thm1}, it is enough to show
\begin{equation}\label{eq3}
\aFsumConn (r)=0  \quad (\forall r \in R_{10000} (b, d))
\end{equation}
for a proof of Corollary~\ref{cor3}.  We first note that if $x_i \in \irr$, then $x_i$ has no isolated chords.  On the other hand, let $r \in R_{10000}(b, d)$, that is, there exist a Gauss word $S$ and a letter $j$ such that $r=[Sjj]$.  Then, the chord corresponding to $j$ is isolated.  
These show that $\tilde{x}_i (r)=0$.  This shows that (\ref{eq3}) holds.  

Note that $\FsumConn$ is a special case of $\Njsum \alpha_j x_j$.  Thus, we may suppose that integers $\alpha_j $ ($n_{b-1} + 1 \le j \le n_{d}$) satisfy that $\FsumConn$ $=$  $\Njsum \alpha_j x_j$.  Therefore, if $\aFsumConn (r)$ $=0$, then $\Njsum \alpha_j \tilde{x}_j (r)$ $=0$.  By Theorem~\ref{gg_thm1}, $\sum_{n_{b-1} + 1 \le j \le n_{d}} \alpha_j x_j$ is invariant.  Since each non-zero term is common, $\FsumConn$ is also invariant.

Further, by using geometric observations as in Example~\ref{example1} it is clear that if $x_i \in \conn$, then \begin{equation}\label{eq4a}
{x_i}(P_1 \sharp_{(p_1,~p_2), h} P_2) = {x_i}(P_1) + {x_i}(P_2)
\end{equation}
for any $P_1 \sharp_{(p_1, p_2), h} P_2$.  This fact implies $\FsumConn$ is additive.  
$\hfill\Box$

\section{Proof of Theorem~\ref{gg_thm1}}\label{proof_main_thms}
Since every Reidemeister move on plane curves may be regarded as that of spherical curves, we will prove the statement for  Reidemeister moves on spherical curves in the following.    

$\bullet$ (Proof of the case $\epsilon_1 = 1$.)
Let $P$ and $P'$ be two spherical  curves where $P$ and $P'$ are related by a single RI, hence, there exist a letter $i$ and two Gauss words $G=Sii$ and $G'=S$ corresponding to $P$ and $P'$, respectively, i.e., $CD_P$ $=$ $[Sii]$ and $CD_{P'}$ $=$ $[S]$.
As we observed in a paragraph preceding of Definition~\ref{dfn_relator1} in  Subsection~\ref{rr1}, we have
\[\Fsum(P) -\Fsum(P')
= \sum_{n_{b-1} + 1 \le i \le n_d} \sum_{z_0 \in \sub^{(0)}(G)} \alpha_i \tilde{x}_i \left( [z_0 ii] \right).
\]
By the assumption of this case, for each $z_0 \in \sub^{(0)}(G)$, 
\[
\sum_{n_{b-1} + 1 \le i \le n_d} \alpha_i \tilde{x}_i \left( [z_0 ii] \right)=0
\]
and this shows that 
\[
\Fsum(P) = \Fsum(P').  
\]
Hence, $\Fsum$ is invariant  under RI.  

$\bullet$ (Proof of the case $\epsilon_2 =1$.)  Let $P$ and $P'$ be two spherical curves where $P$ and $P'$ are related by a single strong~RI\!I, hence, there exist two Gauss words $G=SijTji$ and $G'=ST$ corresponding to $P$ and $P'$, respectively, i.e., $CD_P$ $=$ $[SijTji]$ and $CD_{P'}$ $=$ $[ST]$. 

By (\ref{tilde_x}) in Subsection~\ref{section_def1} and (\ref{sub}) in Subsection~\ref{rr1}, we obtain (note that $\sub^{(3)} (G)$ $=$ $\emptyset$): 
  
\begin{align*}
&\Fsum (P) = \sFsum \left(\sum_{z \in \sub(G)} \tilde{x}_i (z) \right) \\
&= \sFsum \left(  \sum_{z_0 \in \sub^{(0)}(G)} \tilde{x}_i (z_0) \right)
+ \Nsum \sum_{z_{12} \in \sub^{(1)}(G) \cup \sub^{(2)} (G)} \alpha_i \tilde{x}_i (z_{12})\\
&= \sFsum \left(  \sum_{z' \in \sub(G')} \tilde{x}_i (z') \right)
+ \Nsum \sum_{z_{12} \in \sub^{(1)}(G) \cup \sub^{(2)} (G)} \alpha_i \tilde{x}_i (z_{12})
\end{align*}
($\because$ $\sub^{(0)}(G)$ is identified with $\sub(G')$).  

Let $z_0 \in \sub^{(0)}(G)$.
We note that since $G$ is a Gauss word $z_0$ uniquely admits a decomposition into two sub-words, which are sub-words on $S$ and $T$. 
Let $\sigma(z_0)$ be the sub-word of $S$ and $\tau(z_0)$ the sub-word of $T$ satisfying $z_0$ $=$ $\sigma(z_0)\tau(z_0)$.  Under these notations, we define maps
\begin{align*}
z_2: &\sub^{(0)}(G) \to \sub^{(2)}(G); z_2 (z_0) = \sigma(z_0)ij\tau(z_0)ji,\\
z_1: &\sub^{(0)}(G) \to \sub^{(1)}(G); z_1 (z_0) = \sigma(z_0)i\tau(z_0)i,~{\textrm{and}}~\\
z'_1: &\sub^{(0)}(G) \to \sub^{(1)}(G); z'_1 (z_0) = \sigma(z_0)j\tau(z_0)j.  
\end{align*}
Then, it is easy to see that $\sub^{(1)} (G) \cup \sub^{(2)} (G)$ admits a decomposition 
\begin{align*}
&\sub^{(1)} (G) \cup \sub^{(2)} (G) \\
&=  \{ z_1 (z_0) ~|~ z_0 \in \sub^{(0)} (G) \} \amalg \{ z'_1 (z_0) ~|~ z_0 \in \sub^{(0)} (G) \} \amalg \{ z_2 (z_0) ~|~ z_0 \in \sub^{(0)}(G) \}.  
\end{align*}
These notations together with the above give: 
\begin{align*}
\Fsum(P) &= \Fsum (P') \\
&\qquad + \Nsum \sum_{z_0 \in \sub^{(0)}(G)} \at (z_1 (z_0) + z'_1 (z_0) + z_2 (z_0))\\
&= \Fsum (P') \\
&\qquad + \Nsum \sum_{z_0 \in \sub^{(0)}(G)} \at ([z_1 (z_0)] + [z'_1 (z_0)] + [z_2 (z_0)]).
\end{align*}
By Proposition~\ref{relator_prop} and by the condition for the case $\epsilon_2$ $=$ $1$, 
for any $z_0 \in \sub^{(0)}(G)$, 
\[
\aFsum ([z_1 (z_0)] + [z'_1 (z_0)] + [z_2 (z_0)]) = 0.
\]  

Thus,  
\begin{align*}
\Fsum (P) = \Fsum (P').
\end{align*}
Hence, $\Fsum$ is invariant  under strong~RI\!I.

$\bullet$ (Proof of the case $\epsilon_3 =1$.)  
Since the arguments are essentially the same as those for the case $\epsilon_2$ $=$ $1$, we omit this proof.  

$\bullet$ (Proof of the case $\epsilon_4 =1$.)  Let $P$ and $P'$ be two spherical curves where $P$ and $P'$ are related by a single strong~RI\!I\!I, hence, there exist two Gauss words $G=SijTkiUjk$ and $G'=SjiTikUkj$ corresponding to $P$ and $P'$, respectively, i.e., $CD_P$ $=$ $[SijTkiUjk]$ and $CD_{P'}$ $=$ $[SjiTikUkj]$.   

By (\ref{tilde_x}) in Subsection~\ref{section_def1} and (\ref{sub}) in Subsection~\ref{rr1}, we obtain: 
\begin{align*}
&\Fsum (P) = \sFsum \left( \sum_{z \in \sub(G)} \tilde{x}_i (z) \right)  \\
&= \sFsum \left( \sum_{z_{01} \in \sub^{(0)}(G) \cup \sub^{(1)}(G)} \tilde{x}_i (z_{01}) 
+ \sum_{z_{23} \in \sub^{(2)}(G) \cup \sub^{(3)}(G)} \tilde{x}_i (z_{23})
 \right).  
\end{align*}
and
\begin{align*}
&\Fsum (P') = \sFsum \left( \sum_{z' \in \sub(G')} \tilde{x}_i (z') \right)  \\
&= \sFsum \left( \sum_{z'_{01} \in \sub^{(0)}(G') \cup \sub^{(1)}(G')} \tilde{x}_i (z'_{01}) 
+ \sum_{z'_{23} \in \sub^{(2)}(G') \cup \sub^{(3)}(G')} \tilde{x}_i (z'_{23})
 \right).
\end{align*}

Since $\sub^{(0)}(G)$ ($\sub^{(1)}(G)$ resp.) is naturally identified with $\sub^{(0)}(G')$ ($\sub^{(1)}(G')$ resp.), the above equations show: 
\begin{align*}
&\Fsum (P) - \Fsum (P') \\
&= \Nsum \sum_{z_{23} \in \sub^{(2)}(G) \cup \sub^{(3)} (G)} \at (z_{23})\\
&\qquad - \Nsum \sum_{z'_{23} \in \sub^{(2)}(G') \cup \sub^{(3)} (G')} \at (z'_{23}).   
\end{align*}

Let $z_0 \in \sub^{(0)}(G)$, which is identified with $\sub^{(0)}(G')$.  
We note that since $G$ is a Gauss word $z_0$ uniquely admits a decomposition into three sub-words, which are sub-words on $S$, $T$, and $U$. 
Let $\sigma(z_0)$ be the sub-word of $S$, $\tau(z_0)$ the sub-word of $T$, and $\mu(z_0)$ the sub-word of $U$ satisfying $z_0$ $=$ $\sigma(z_0)\tau(z_0)\mu(z_0)$.  
We define maps   
\begin{align*}
z_3: &\sub^{(0)}(G) \to \sub^{(3)}(G); z_3 (z_0) = \sigma(z_0) ij \tau(z_0) ki \mu(z_0) jk,\\ 
z_{2a}: &\sub^{(0)}(G) \to \sub^{(2)}(G); z_{2a} (z_0) = \sigma(z_0) ij \tau(z_0) i \mu(z_0) j, \\
z_{2b}: &\sub^{(0)}(G) \to \sub^{(2)}(G); z_{2b} (z_0) = \sigma(z_0) i \tau(z_0) ki \mu(z_0) k,~{\text{and}}~\\
z_{2c}: &\sub^{(0)}(G) \to \sub^{(2)}(G); z_{2c} (z_0) = \sigma(z_0) j \tau(z_0) k\mu(z_0) jk.
\end{align*}
Similarly, let 
\begin{align*}
z'_3: &\sub^{(0)}(G') \to \sub^{(3)}(G'); z'_3 (z_0) = \sigma(z_0) ji \tau(z_0) ik \mu(z_0) kj,\\
z'_{2a}: &\sub^{(0)}(G') \to \sub^{(2)}(G'); z'_{2a} (z_0) = \sigma(z_0) ji \tau(z_0) i \mu(z_0) j, \\
z'_{2b}: &\sub^{(0)}(G') \to \sub^{(2)}(G'); z'_{2b} (z_0) = \sigma(z_0) i \tau(z_0) ik \mu(z_0) k,~{\text{and}}~ \\
z'_{2c}: &\sub^{(0)}(G') \to \sub^{(2)}(G'); z'_{2c} (z_0) = \sigma(z_0) j \tau(z_0) k \mu(z_0) kj.
\end{align*}
Then, it is easy to see that $\sub^{(2)} (G) \cup \sub^{(3)} (G)$ admits decompositions 
\begin{align*}
&\sub^{(2)} (G) \cup \sub^{(3)} (G) \\
&= \{ z_3 (z_0) ~|~ z_0 \in \sub^{(0)}(G) \}
\amalg \{ z_{2a} (z_0) ~|~ z_0 \in \sub^{(0)} (G) \} 
\amalg \{ z_{2b} (z_0) ~|~ z_0 \in \sub^{(0)} (G) \}  \\
& \amalg \{ z_{2c} (z_0) ~|~ z_0 \in \sub^{(0)}(G) \}
\end{align*}
and 
\begin{align*}
& \sub^{(2)}(G') \cup \sub^{(3)} (G') \\
&= \{ z'_3 (z_0) ~|~ z_0 \in \sub^{(0)}(G) \}
\amalg \{ z'_{2a} (z_0) ~|~ z_0 \in \sub^{(0)} (G) \} 
 \amalg \{ z'_{2b} (z_0) ~|~ z_0 \in \sub^{(0)} (G) \} \\
&\amalg \{ z'_{2c} (z_0) ~|~ z_0 \in \sub^{(0)} (G) \}.
\end{align*}
By using these notations, we obtain: 
\begin{align*}
&\Fsum (P) - \Fsum (P')
\\
&= \sum_{z_0 \in \sub^{(0)}(G)}  \aFsum \Big( (z_3 (z_0) + z_{2a} (z_0) + z_{2b} (z_0) + z_{2c} (z_0))\\ 
&- (z'_3(z_0) + z'_{2a}(z_0) + z'_{2b}(z_0) + z'_{2c}(z_0)) \Big) \\
&= \sum_{z_0 \in \sub^{(0)}(G)} \aFsum \Big(  ([z_3 (z_0)] + [z_{2a} (z_0)] + [z_{2b} (z_0)] + [z_{2c} (z_0)])\\ 
& - ([z'_3(z_0)] + [z'_{2a}(z_0)] + [z'_{2b}(z_0)] + [z'_{2c}(z_0)]) \Big).
\end{align*}
Here, we note that 
\[
([z_3 (z_0)] + [z_{2a} (z_0)] + [z_{2b} (z_0)] + [z_{2c} (z_0)]) - ([z'_3(z_0)] + [z'_{2a}(z_0)] + [z'_{2b}(z_0)] + [z'_{2c}(z_0)]) \in R_{00010}.  
\]
Hence, by the assumption of Case $\epsilon_4$ $=$ $1$ and by Proposition~\ref{relator_prop}, for each $z_0$, 
\begin{align*}
&\aFsum  \Big(  ([z_3 (z_0)] + [z_{2a} (z_0)] + [z_{2b} (z_0)] + [z_{2c} (z_0)]) \\
&\qquad\qquad\qquad  - ([z'_3(z_0)] + [z'_{2a}(z_0)] + [z'_{2b}(z_0)] + [z'_{2c}(z_0)]) \Big)\\
&\qquad = 0.  
\end{align*}
These show that  
\[ \Fsum (P) = \Fsum (P').
\]
Hence, $\Fsum$ is invariant  under strong~RI\!I\!I.

$\bullet$ (Proof of the case $\epsilon_5 =1$.)  
Since the arguments are essentially the same as those for the case $\epsilon_4$ $=$ $1$, we omit this proof.

$\hfill\Box$

\section{Complete analysis of invariants obtained from ${G}_{2, 4}$ and ${R}_{00010} (2, 4)$}\label{sec_appl}
In this section, by using Corollary~\ref{cor3}, we completely analyze invariants obtained from ${G}_{2, 4} \cap \conn$ and ${R}_{00010} (2, 4)$.  First, it is easy to see that $\{ x_i \}_{i \in I^{(\conn)}_{2, 4}}$ consists of nine elements, 
$y_1$ $=$ $\x$, 
$y_2$ $=$ $\tr$, 
$y_3$ $=$ $\h$, 
$y_4$ $=$ $\sh$,  
$y_5$ $=$ $\hthc$,
$y_6$ $=$ $\htri$,
$y_7$ $=$ $\qua$,  
$y_8$ $=$ $\crh$, and
$y_9$ $=$ $\ma$.  
This fact together with Corollary~\ref{cor3} implies that if there exist integers $\alpha_i$ ($1 \le i \le 9$) such that $(\sum_{i=1}^9 \alpha_i \tilde{y}_i)(r)$ $=$ $0$ ($\forall r \in {R}_{00010}(2, 4)$), then $\sum_{i=1}^9 \alpha_i y_i$ is invariant under RI and strong RI\!I\!I. 
Second, recall that $R_{00010}(2, 4)$ consists of twelve elements $r_j$ ($1 \le j \le 12$) as follows (see Example~\ref{ex1_relator34}):  
\begin{align*}
r_1 &= \tr + 3 \x - 3 \chordth - \tri, \\
r_2 &= \triso + 3 \chordtwo - \chordthab - 2 \thc - \tri, \\
r_3 &= \ma + \tr + \h -\trii - 2 \chordtwo, \\
r_4 &= 2 \htri - \crh - \chordtwoc, \\
r_5 &= \ma + 2 \hthc - \sh - 2 \triii, \\
r_6 &= \qua + 2 \htri - \ma - 2 \triso, \\ 
r_7 &= 2 \ma + \sh - 2 \htri - \crh, \\
r_{8} &= \htri + 2 \crh -  \hthc - 2 \trii,\\
r_{9} &= 3 \chordtwoc - 2 \chordtha - \chordthb, \\
r_{10} &= 3 \chordtha - 2 \four - \chordthab, \\
r_{11} &= 3 \chordthb - 2 \chordthab - \triFour,~{\text{and}}~\\
r_{12} &= \trisox + 2 \chordtwx - \four - 2 \chordthab.\\
\end{align*}
\begin{figure}[b]
\includegraphics[width=12cm]{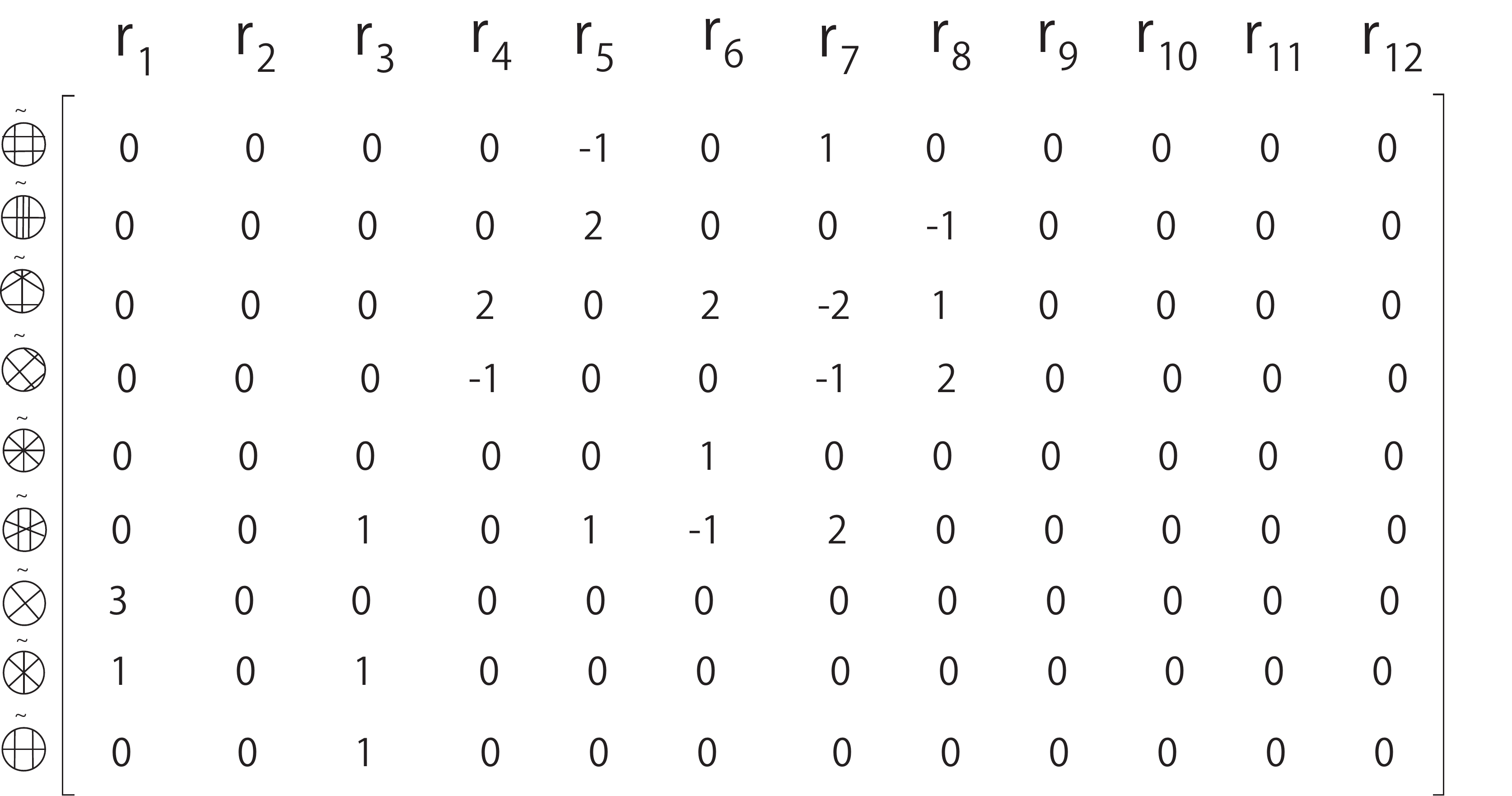}
\caption{Matrix $(\tilde{y}_i (r_j))_{1 \le i \le 9, 1 \le j \le 12}.$}\label{table1}
\end{figure}
We can find $\alpha_i$ ($1 \le i \le 9$) satisfying 
$(\sum_{i=1}^9 \alpha_i \tilde{y}_i)(r)$ $=$ $0$ ($\forall r \in {R}_{00010}(2, 4)$)
by solving the linear equation ${\bf{x}} M$ $=$ $\bf{0}$, where $M$ $=$ $(\tilde{y}_i (r_j) )$ (Figure~\ref{table1}).  It is elementary to show that the set of solutions is $\{ {\bf{x}}$ $=$ $\gamma_1 (1, -3, 3, 0, 0, 0, 0, 0, 0)$, $\gamma_2 (0, 0, 2, 8, 5, 1, -4, 2, -2)$ $|~\gamma_1, \gamma_2 \in \mathbb{Z}\}$.  
Particularly, $y_1$ $-3 y_2$ $+$ $3 y_3$ ($=$ $\x$ $-3 \tr$ $+ 3 \h$) and 
$2 y_3$ $+$ $8 y_4$ $+$ $5 y_5$ $+$ $y_6$ $-$ $4 y_7$ $+$ $2 y_8$ $-$ $2 y_9$ ($=$ $2 \h$ $+$ $8 \sh$ $+$ $5 \hthc$ $+$ $\htri$ $- 4 \qua$ $+$ $2 \crh$ $- 2 \ma$)
 are invariant under RI and strong RI\!I\!I.  
 
Let $\lambda^{(3)}$ $=$ $y_1$ $-3 y_2$ $+$ $3 y_3$ and let $\lambda^{(4)}$ $=$ $2 y_3$ $+$ $8 y_4$ $+$ $5 y_5$ $+$ $y_6$ $-$ $4 y_7$ $+$ $2 y_8$ $-$ $2 y_9$.  Note that $\lambda^{(3)}$ is the function obtained in Example~\ref{ex_00010}, hence $\lambda^{(3)}$ $=$ $4 \lambda$ where $\lambda$ is the function that appears in Theorem~1 of \cite{IT3}.    
We further note that in \cite{IT3}, it is shown that $\lambda$ (hence, $\frac{1}{4} \lambda^{(3)}$) is a complete invariant of prime spherical  curves up to seven double points (Table~\ref{t1}).    
We note that $\lambda^{(4)}$ is also a complete invariant of prime spherical curves up to seven double points.  
We further note that $\lambda^{(4)}$ is stronger than $\lambda^{(3)}$ in some sense.  In fact, $\lambda (4_1)$ $=$ $\lambda(8_9)$ $=$ $4$ and $\lambda^{(4)} (4_1)$ ($= 16$) $\neq$ $\lambda^{(4)}(8_9)$ $(= 40)$.  Similarly, we have $\lambda(5_1)$ $=$ $\lambda (8_2)$ $=$ $-5$ and $\lambda^{(4)}(5_1)$ ($= -20$) $\neq$ $\lambda^{(4)} (8_2)$ $(= -28)$, and $\lambda(8_5)$ $=$ $\lambda(8_{17})$ $=$ $12$ and $\lambda^{(4)}(8_5)$ ($= 128$) $\neq$ $\lambda^{(4)}(8_{17})$ $(= 116)$.

\begin{table}[h!]
\caption{Values of $\frac{1}{4}\lambda^{(3)}$ ($=$ $\lambda$) and $\lambda^{(4)}$ for prime spherical curves with at most seven double points.  Any pair of prime spherical curves in the same box in the leftmost column are related by a finite sequence of Reidemeister moves of types RI and strong~RI\!I\!I.}\label{t1}
\includegraphics[width=12cm]{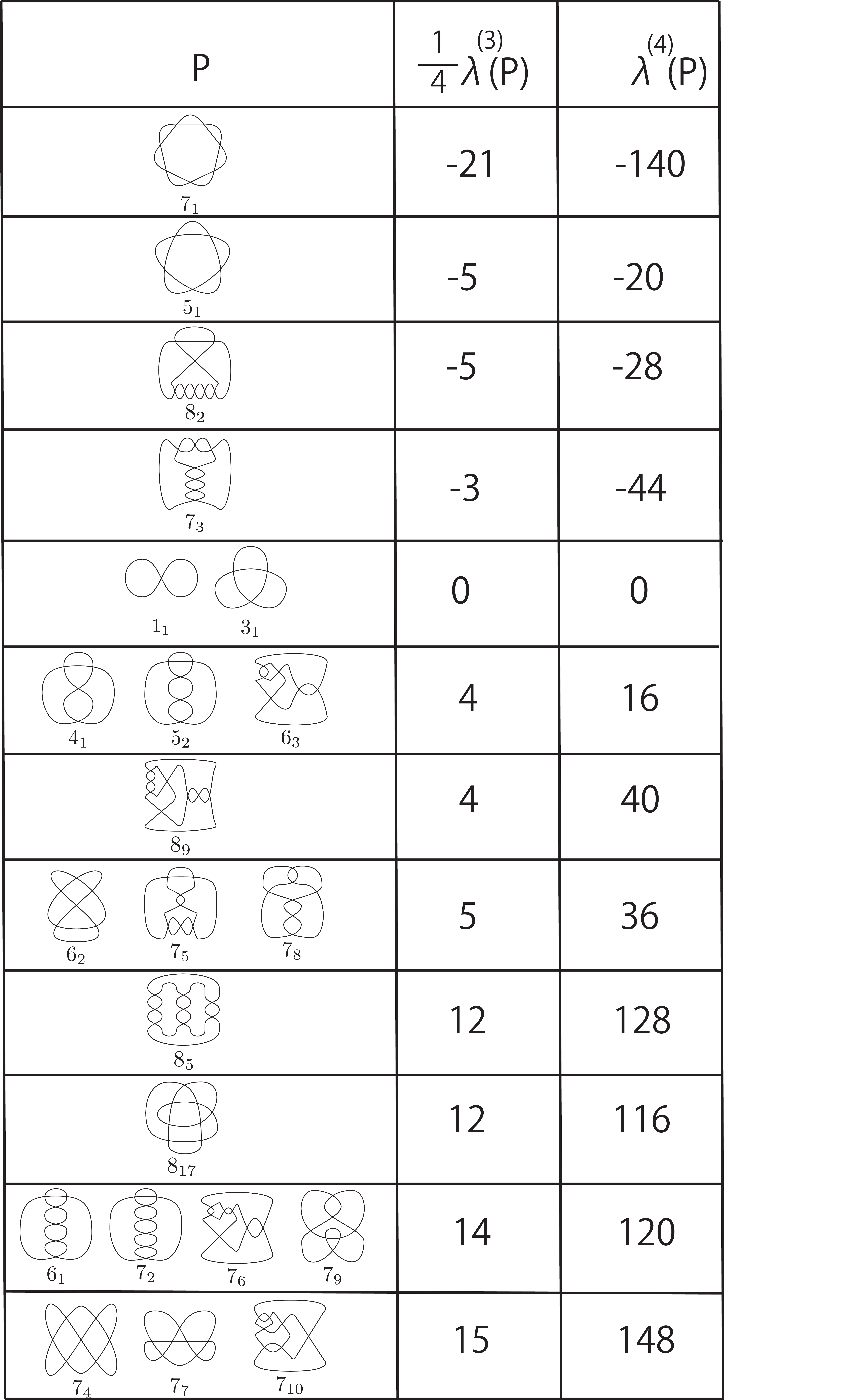}
\end{table}

\section*{Acknowledgements}
The author would like to thank Professors Tsuyoshi Kobayashi and Kouki Taniyama for helpful discussions.  The author also would like to thank Mr.~Yusuke Takimura for creating and providing many figures of spherical curves with indices (Table~\ref{t1}) and for useful discussions.  Further, the author would like to thank Professor Yukari Funakoshi and Ms.~Megumi Hashizume for their comments on an earlier version of this paper.   The author would like to thank the referee for the comments.  
The author was a project researcher of Grant-in-Aid for Scientific Research (S) 24224002 (April 2016--March 2017).

\end{document}